\documentclass[11pt]{article}

\usepackage{epsfig,epsf,fancybox}
\usepackage{amsmath}
\usepackage{mathrsfs}
\usepackage{amssymb}
\usepackage{graphicx}
\usepackage{color}
\usepackage{multirow}
\usepackage{paralist}
\usepackage{verbatim}
\usepackage{galois}
\usepackage{algorithm}
\usepackage{algorithmic}
\usepackage{boxedminipage}
\usepackage{booktabs}
\usepackage{accents}
\usepackage{stmaryrd}
\usepackage{subcaption}
\usepackage{wrapfig}
\usepackage[bottom]{footmisc}
\usepackage{natbib}

\usepackage[colorlinks,linkcolor=magenta,citecolor=blue, pagebackref=true,backref=true]{hyperref}
\renewcommand*{\backrefalt}[4]{%
    \ifcase #1 \footnotesize{(Not cited.)}%
    \or        \footnotesize{(Cited on page~#2.)}%
    \else      \footnotesize{(Cited on pages~#2.)}%
    \fi}

\long\def\comment#1{}
\usepackage{nicefrac}

\newtheorem{theorem}{Theorem}[section]

\newtheorem{lemma}[theorem]{Lemma}
\newtheorem{proposition}[theorem]{Proposition}

\newtheorem{definition}{Definition}[section]
\newtheorem{example}{Example}[section]

\newtheorem{remark}{Remark}[section]

\newtheorem{inftheorem}{Informal Theorem}[section]

\numberwithin{equation}{section}

\newcommand{\x}{\mathbf x}

\newcommand{\s}{\mathbf s}

\newcommand{\argmin}{\mathop{\rm argmin}}

\newcommand{\LCal}{\mathcal{L}}

\newcommand{\CCal}{\mathcal{C}}

\newcommand{\KCal}{\mathcal{K}}

\newcommand{\SCal}{\mathcal{S}}
\newcommand{\TCal}{\mathcal{T}}

\newcommand{\ICal}{\mathcal{I}}

\newcommand{\br}{\mathbb{R}}

\newcommand{\ba}{\begin{array}}
\newcommand{\ea}{\end{array}}

\newcommand{\RCal}{\mathcal{R}}
\newcommand{\PCal}{\mathcal{P}}

\newcommand{\NCal}{\mathcal{N}}

\newcommand{\mydefn}{:=}

\begin{document}

\begin{center}

{\bf{\LARGE{First-Order Algorithms for Nonlinear Generalized Nash \\ [.2cm] Equilibrium Problems}}}

\vspace*{.2in}
{\large{
\begin{tabular}{c}
Michael I. Jordan$^{\diamond, \dagger}$ \quad Tianyi Lin$^\diamond$ \quad Manolis Zampetakis$^\diamond$
\end{tabular}
}}

\vspace*{.2in}

\begin{tabular}{c}
Department of Electrical Engineering and Computer Sciences$^\diamond$ \\
Department of Statistics$^\dagger$ \\ 
University of California, Berkeley
\end{tabular}

\vspace*{.2in}

\today

\vspace*{.2in}

\begin{abstract}
We consider the problem of computing an equilibrium in a class of \textit{nonlinear generalized Nash equilibrium problems (NGNEPs)} in which the strategy sets for each player are defined by equality and inequality constraints that may depend on the choices of rival players. While the asymptotic global convergence and local convergence rates of algorithms to solve this problem have been extensively investigated, the analysis of nonasymptotic iteration complexity is still in its infancy. This paper presents two first-order algorithms---based on the quadratic penalty method (QPM) and augmented Lagrangian method (ALM), respectively---with an accelerated mirror-prox algorithm as the solver in each inner loop. We establish a global convergence guarantee for solving monotone and strongly monotone NGNEPs and provide nonasymptotic complexity bounds expressed in terms of the number of gradient evaluations. Experimental results demonstrate the efficiency of our algorithms in practice. 
\end{abstract}
\end{center}

\section{Introduction}
The Nash equilibrium problem (NEP)~\citep{Nash-1950-Equilibrium, Nash-1951-Noncooperative} is a central topic in mathematics, economics and computer science. NEP problems have begun to play an important role in machine learning as researchers begin to focus on decisions, incentives and the dynamics of multi-agent learning. In a classical NEP, the payoff to each player depends upon the strategies chosen by all, but the domains from which the strategies are to be chosen for each player are independent of the strategies chosen by other players. The goal is to arrive at a joint optimal outcome where no player can do better by deviating unilaterally~\citep{Osborne-1994-Course, Myerson-2013-Game}. 

The \emph{generalized Nash equilibrium problem} (GNEP) is a natural generalization of an NEP where the choice of an action by one agent affects both the payoff and the \textit{domain} of actions of all other agents~\citep{Arrow-1954-Existence}. Its introduction in the 1950's provided the foundation for a rigorous theory of economic equilibrium~\citep{Debreu-1952-Social, Arrow-1954-Existence, Debreu-1959-Theory}. More recently, the GNEP problem has emerged as a powerful paradigm in a range of engineering applications involving noncooperative games. In particular, in the survey of~\citet{Facchinei-2010-Generalized}, three general classes of problems were developed in detail: the abstract model of general equilibrium, power allocation in a telecommunication system, and environmental pollution control. Further applications of the GNEP problem in recent years have included adversarial classification~\citep{Bruckner-2009-Nash, Bruckner-2012-Static}, wireless communication and networks~\citep{Pang-2008-Distributed, Pang-2010-Design, Han-2012-Game, Scutari-2014-Real}, power grids~\citep{Jing-1999-Spatial, Hobbs-2007-Nash}, cloud computing~\citep{Ardagna-2011-Game, Ardagna-2015-Generalized}, modern traffic systems with e-hailing services~\citep{Ban-2019-General}, supply and demand constraints for transportation systems~\citep{Stein-2018-Noncooperative} and pollution quotas for environmental applications~\citep{Krawczyk-2005-Coupled, Breton-2006-Game}.  For an overview of GNEP theory and applications, we refer to more detailed surveys~\citep{Cominetti-2012-Modern,Facchinei-2007-Finite, Facchinei-2010-Generalized, Fischer-2014-Generalized} and the references therein.

It is of significant interest to bring the GNEP framework and its decision-making applications into contact with machine learning. By analogy with the bridge between machine learning and smooth nonlinear optimization that has been so productive in the recent years, the major challenges that arise for GNEP formulations of machine-learning problems are computational, including the development of scalable gradient-based algorithms. The first attempts to design such algorithms are due to~\citet{Robinson-1993-Shadow-I, Robinson-1993-Shadow-II}. Further progress has been made over the ensuing three decades by considering generic algorithms based on either the Nikaido-Isoda function~\citep{Uryas-1994-Relaxation, Krawczyk-2000-Relaxation, Von-2009-Relaxation, Facchinei-2009-Generalized, Dreves-2011-Solution, Von-2012-Newton, Dreves-2013-Globalized, Izmailov-2014-Error, Fischer-2016-Globally} or penalty functions~\citep{Pang-2005-Quasi, Facchinei-2011-Partial, Fukushima-2011-Restricted, Facchinei-2010-Penalty, Kanzow-2016-Multiplier, Kanzow-2016-Augmented, Kanzow-2018-Augmented, Ba-2022-Exact}. In terms of theoretical guarantees, global convergence and local convergence rates have been established for some of these algorithms under suitable assumptions. For an overview of recent progress on penalty-type algorithms, we refer to~\citet{Ba-2022-Exact}. 

Despite this progress on the algorithmic aspects of GNEPs, significant computational challenges remain and they hinder the the practical impact of the GNEP concept in machine learning. Most notably, global convergence rate characterizations are not yet available for any of the algorithms that target the solution of nonlinear GNEPs. Thus, we summarize the focus of our work: \textit{Can we design efficient algorithms for solving GNEPs that have global convergence rate guarantees?} What rates can we obtain and in what class of GNEPs?

\subsection{Contribution}
We tackle the problem of designing efficient first-order algorithms for a class of nonlinear GNEPs and proving optimal global convergence rates in various settings. More specifically, we start by defining a class of \textit{monotone and strongly monotone} GNEPs which cover a range of machine-learning applications. Leveraging the recent progress on the iteration complexity analysis of first-order algorithms for nonlinear optimization and variational inequality (VI) problems~\citep{Lan-2013-Iteration, Lan-2016-Iteration, Chen-2017-Accelerated, Xu-2021-Iteration}, we develop first-order algorithms for computing the solutions of GNEPs with global convergence rate estimates. The algorithms that we study incorporate the accelerated mirror-prox (AMP) scheme into a quadratic penalty method (QPM) or an augmented Lagrangian method. At a high level, the following informal theorems summarize the main results of our paper:
\begin{inftheorem}[Theorem \ref{Thm:AMPQP}]
The required number of gradient evaluations for accelerated mirror-prox quadratic penalty (AMPQP) method to reach an $\epsilon$-solution is bounded by $\tilde{O}(\epsilon^{-1})$ and $\tilde{O}(\epsilon^{-1/2})$ in monotone and strongly monotone NGNEPs.
\end{inftheorem}
\begin{inftheorem}[Theorem \ref{Thm:AMPAL}]
The required number of gradient evaluations for the accelerated mirror-prox augmented Lagrangian (AMPAL) method to reach an $\epsilon$-solution is bounded by $\tilde{O}(\epsilon^{-1})$ and $\tilde{O}(\epsilon^{-1/2})$ in monotone and strongly monotone NGNEPs.
\end{inftheorem}
Although the algorithmic frameworks based on QPM and ALM are classical for GNEPs, we remark that combining these frameworks with AMP is new in the literature. For inexact ALM, the convergence guarantees are asymptotic for GNEPs while the rate results are only derived for an optimization setting when combined with Nesterov's accelerated gradient (NAG) method. These analyses cannot be directly extended to NGNEPs since NAG does not apply to the VI subproblem. 

Building on background in linearly constrained optimization~\citep{Lan-2013-Iteration, Lan-2016-Iteration}, our results are new in the GNEP setting. In particular, we highlight our analysis of the AMP algorithm for strongly monotone VIs and the demonstration of its applicability for NGNEPs, which is established via the derivation of an optimal convergence rate in monotone settings. We also prove the equivalence between KKT points and solutions for NGNEPs.  Noting that~\citet{Bueno-2019-Optimality} have presented a counterexample for general GNEPs, our result is obtained by identifying and exploiting a special structure of NGNEPs. This equivalence forms the basis of our development of first-order algorithms for solving NGNEPs.

We make some further comments. Note that there are many instances of GNEPs in specialized problems where global convergence rate guarantees have been established~\citep{Nesterov-2011-Solving, Nabetani-2011-Parametrized, Bianchi-2022-Fast, Franci-2022-Stochastic}. These GNEPs are either monotone or satisfy an error bound condition that is analogous to strong monotonicity. This suggests that the notion of monotonicity or strongly monotonicity is key to global convergence rate estimates for the algorithms. Further, the local convergence rates were derived under local strong monotonicity or more general local error bounds~\citep{Facchinei-2009-Generalized, Facchinei-2015-Semismooth}. There are also several natural applications for monotone GNEPs, motivating the investigation of \textit{simple}, \textit{optimal} and \textit{implementable} first-order algorithms for monotone GNEPs. In particular, the real-world GNEPs application problems arising from machine learning, e.g., team games~\citep{Celli-2018-Computational, Celli-2019-Coordination, Farina-2021-Connecting, Carminati-2022-Marriage, Kalogiannis-2022-Efficiently}, are extremely large and make significant demands with respect to computational feasibility. This necessitates the investigation of first-order algorithms for GNEPs. In summary, although the notion of global monotonicity or strong monotonicity rules out some interesting application problems, it encompasses a rather rich class of GNEPs and leads to global convergence rate guarantees. 

\subsection{Related work} 
To appreciate the broad scope of our agenda, we start by reviewing the algorithms for computing the solution of GNEPs.~\citet{Rosen-1965-Existence} studied the jointly convex GNEP and proposed an algorithm based on gradient projection. Subsequently, other algorithms were developed, including the relaxation method~\citep{Uryas-1994-Relaxation, Krawczyk-2000-Relaxation, Von-2009-Relaxation}, Newton-type methods~\citep{Facchinei-2009-Generalized, Von-2012-Newton, Dreves-2013-Globalized, Izmailov-2014-Error, Fischer-2016-Globally} and the interior-point potential reduction method~\citep{Dreves-2011-Solution}. Most of these approaches are based on the Nikaido-Isoda function and thus are very expensive computationally. By exploiting the special structure of certain GNEPs, some specific algorithms were developed with improved theoretical guarantees; these include the VI approaches of~\citet{Facchinei-2007-Generalized} and~\citet{Nabetani-2011-Parametrized}, which are restricted to jointly convex GNEPs, and Lemke's method from~\citet{Schiro-2013-Solution}, which is specifically designed to solve a class of linear GNEPs.

More recently, progress has been made on developing penalty-type algorithms for GNEPs.  Representatives of this class of algorithms are exact and inexact penalty methods~\citep{Facchinei-2011-Partial, Fukushima-2011-Restricted, Facchinei-2010-Penalty, Kanzow-2018-Augmented, Ba-2022-Exact} and exact and inexact augmented Lagrangian method~\citep{Pang-2005-Quasi, Kanzow-2016-Multiplier, Kanzow-2016-Augmented, Kanzow-2018-Augmented}. The exact penalty method was introduced in~\citet{Fukushima-2011-Restricted}, and its variants have been proposed where either all of the constraints are penalized~\citep{Facchinei-2010-Penalty} or some of the constraints are penalized~\citep{Facchinei-2011-Partial}. These algorithms achieve exactness results under suitable assumptions but suffer from the nonsmoothness of penalized subproblems, thereby leading to numerical difficulties from a practical point of view.~\citet{Pang-2005-Quasi} proposed the augmented Lagrangian method to solve quasi-variational inequalities (QVIs) and~\citet{Kanzow-2016-Multiplier} improved this scheme with a global convergence guarantee. They discussed the GNEP within their framework by treating it as a special QVI. Later on,~\citet{Kanzow-2016-Augmented, Kanzow-2018-Augmented} applied a similar idea directly to GNEPs and proved theoretical results which are stronger than those that arise from the QVI framework. Even though an asymptotic global convergence has been established under suitable assumptions, the nonasymptotic global convergence rate (aka, the iteration complexity bound) of these algorithms is open. For a brief overview of algorithmic results, we refer to the survey of~\citet{Fischer-2014-Generalized}. 

Another line of relevant work is concerned with optimality conditions and constraint qualifications (CQs). In nonlinear optimization, the Karush-Kuhn-Tucker (KKT) condition provides a general characterization of local optimality under various CQs~\citep{Kuhn-1951-Nonlinear, Karush-1939-Minima}. Both the KKT condition and CQs have been extended to GNEPs.  In particular,~\citet{Pang-2005-Quasi} considered the generalization of Mangasarian-Fromovitz condition~\citep{Mangasarian-1967-Fritz} for QVIs; see also~\citet{Kanzow-2016-Multiplier}. Subsequently,~\cite{Facchinei-2010-Penalty} specialized this notion to GNEPs and designed exact penalty methods with asymptotic global convergence guarantees.~\citet{Kanzow-2016-Augmented} and~\citet{Bueno-2019-Optimality} studied the augmented Lagrangian methods for GNEPs and established their asymptotic global convergence guarantee under CQs, including the constant positive linear dependence property~\citep{Qi-2000-Constant, Andreani-2005-Relation} and the cone continuity property~\citep{Andreani-2016-Cone}. Despite the significance of these results for theory, their practical impact has been limited.  This is because all of these CQs are defined at or around a solution of the GNEP and it is hard to verify them for a nonlinear GNEP where the candidate solution set is generally unavailable. 

\subsection{Organization}
The remainder of the paper is organized as follows. In Section~\ref{sec:prelim}, we provide the basic setup for a class of nonlinear generalized Nash equilibrium problems (NGNEPs) with motivating examples and an $\epsilon$-solution concept. We also provide new results for the accelerated mirror-prox (AMP) algorithm~\citep{Chen-2017-Accelerated} in the strongly monotone case which complements the current iteration complexity analysis. In Section~\ref{sec:results}, we propose two first-order algorithms and establish the bounds on their iteration complexity for solving monotone and strongly monotone NGNEPs. In Section~\ref{sec:experiments}, we conduct some preliminary numerical experiments and highlight the practical performance of our algorithms. In Section~\ref{sec:conclu}, we conclude the paper.

\section{Preliminaries}\label{sec:prelim}
We start by providing the definitions for monotone generalized Nash equilibrium problems. Moreover, we define an $\epsilon$-solution concept based on a variational characterization of GNEPs and review the accelerated mirror-prox algorithm and its iteration complexity. 

\paragraph{Notation.} We let $[N] = \{1, 2, \ldots, N\}$ and let $\br_+^n$ denote the set of all vectors in $\br^n$ with nonnegative entries. For a vector $x \in \br^n$ and $p \in [1, +\infty]$, we denote $\|x\|_p$ as its $\ell_p$-norm and $\|x\|$ as its $\ell_2$-norm. For a matrix $A \in \br^{n \times n}$, we denote $\lambda_{\max}(A)$ and $\lambda_{\min}(A)$ as largest and smallest eigenvalues and $\|A\|$ as the spectral norm. For a function $f: \br^n \mapsto \br$, we let the subdifferential of $f$ at $x$ be $\partial f(x)$. If $f$ is differentiable, we have $\partial f(x) = \{\nabla f(x)\}$ where $\nabla f(x)$ denotes the gradient of $f$ at $x$ and $\nabla_\nu f$ denotes the partial gradient of $f$ with respect to $\nu$-th block of $x$. Given $n \geq 1$ and desired tolerance $\epsilon \in (0, 1)$, the notation $a = O(b(n, \epsilon))$ stands for $a \leq C \cdot b(n, \epsilon)$ where $C$ is independent of $n$ and $1/\epsilon$, and $a = \tilde{O}(b(n, \epsilon))$ indicates the same inequality where $C$ depends on log factors of $n$ and $1/\epsilon$. 

\subsection{Problem setup}
We consider nonlinear generalized Nash equilibrium problems (NGNEPs) with a finite set of players $\nu \in \NCal = [N]$. The $\nu$-th player selects a strategy $x^\nu$ from a subset $X_\nu(\cdot)$ of a finite-dimensional vector space $\br^{n_\nu}$ and the incurred cost is determined by a function $\theta_\nu(\cdot)$. The profile $x = (x^1, x^2, \ldots, x^N)$ denotes all players' actions and $x^{-\nu}$ denotes the joint action of all players but player $\nu$. We rewrite the overall joint action as $x = (x^\nu; x^{-\nu}) \in \br^n$ where $n = n_1 + n_2 + \ldots + n_N$. 
\begin{definition}\label{def:NGNEPs}
We define a class of NGNEPs by a tuple $G=(\NCal, \{X_\nu(\cdot)\}_{\nu \in \NCal}, \{\theta_\nu(\cdot)\}_{\nu \in \NCal})$, where $X_\nu(\cdot): \br^{n-n_\nu} \rightrightarrows \br^{n_\nu}$ is a point-to-set mapping representing the strategy set of player $\nu$ and $\theta_\nu: \br^n \mapsto \br$ is the $\nu$-th player’s cost function. The following conditions are satisfied:
\begin{enumerate}
\item $\NCal = \cup_{s=1}^S \NCal_s$, where $\NCal_i$ and $\NCal_j$ are not necessarily disjoint for $i \neq j$. Each of $\NCal_s$ is associated with $(A_s, b_s) \in \br^{m_s \times (\sum_{i \in \NCal_s} n_i)} \times \br^{m_s}$ and $(E_s, d_s) \in \br^{e_s \times (\sum_{i \in \NCal_s} n_i)} \times \br^{e_s}$. 
\item $X_\nu(\cdot)$ consists of a simple,\footnote{``Simple" means that the projection operator admits a closed-form solution or can be computed efficiently.} convex and compact set $\hat{X}_\nu$ intersected with the constraints that correspond to $\ICal_\nu = \{s: \nu \in \NCal_s\}$. Indeed, we have $X_\nu(x^{-\nu}) = \{x^\nu \in \hat{X}_\nu: A_s x^{\NCal_s} \leq b_s, \ E_s x^{\NCal_s} = d_s, \ \forall s \in \ICal_\nu\}$ where $x^\NCal$ is the concatenation of all $x^i$ for $i \in \NCal$. 
\item $\theta_\nu(\cdot)$ is continuously differentiable and $\nabla_\nu \theta_\nu(\cdot)$ is $\ell_\theta$-Lipschitz, i.e., we have $\|\nabla_\nu \theta_\nu(x) - \nabla_\nu \theta_\nu(x')\| \leq \ell_\theta \|x - x'\|$ for all $x, x' \in \hat{X}$\footnote{Throughout this paper, we define the base strategy set $\hat{X}$ by $\hat{X} = \Pi_{\nu=1}^N \hat{X}_\nu$.}. 
\end{enumerate}
\end{definition}
\begin{definition}\label{def:NGNEP-solution}
The solution set of an NGNEP contains strategy profiles that discourage unilateral deviations. Formally, $\bar{x} \in \br^n$ is a solution if the following statement holds: 
\begin{equation*}
\theta_\nu(\bar{x}^\nu, \bar{x}^{-\nu}) \leq \theta_\nu(x^\nu, \bar{x}^{-\nu}), \quad \textnormal{for all } x^\nu \in X_\nu(\bar{x}^{-\nu}) \textnormal{ and all } \nu \in \NCal. 
\end{equation*}
\end{definition}
The NGNEP is a generalization of Nash equilibrium problems (NEPs) and Definition~\ref{def:NGNEP-solution} extends the notion of Nash equilibrium~\citep{Rosen-1965-Existence} to the setting where the strategy set of each player depends on the strategies of the rival players through equality and inequality constraints. If the NGNEP is convex, we have a variational characterization for NGNEPs. Indeed, if $\theta_\nu(\cdot, x^{-\nu})$ is convex for any fixed $x^{-\nu}$, an equilibrium can be computed by solving a quasi-variational inequality (QVI)~\citep{Chan-1982-Generalized, Harker-1991-Generalized}. More specifically, letting $v(x) = (v_1(x), \ldots, v_N(x))$ be the profile of all players' individual cost gradients, $v_\nu(x) = \nabla_\nu \theta_\nu(x)$, for all $\nu \in \NCal$, we have the following proposition.  
\begin{proposition}\label{Prop:NGNEP-equivalence}
If the NGNEP is convex, $\bar{x} \in \br^n$ is a solution if and only if $\bar{x}^\nu \in X_\nu(\bar{x}^{-\nu})$ and the following QVI holds: $(x - \bar{x})^\top v(\bar{x}) \geq 0$ for all $x \in X_\nu(\bar{x})$.
\end{proposition}
\begin{proof}
If the QVI holds: $(x - \bar{x})^\top v(\bar{x}) \geq 0$ for all $x \in X_\nu(\bar{x})$, we obtain that $\bar{x}^\nu \in X_\nu(\bar{x}^{-\nu})$ and $(x^\nu - \bar{x}^\nu)^\top v_\nu(\bar{x}) \geq 0$ for all $x^\nu \in X_\nu(\bar{x}^{-\nu})$ by setting $x^{-\nu} = \bar{x}^{-\nu}$. Since the NGNEP is convex, we have $\theta_\nu(\bar{x}^\nu, \bar{x}^{-\nu}) \leq \theta_\nu(x^\nu, \bar{x}^{-\nu})$ for all $x^\nu \in X_\nu(\bar{x}^{-\nu})$. Thus, $\bar{x}$ is a solution. 

Conversely, if $\bar{x}$ is a solution of the NGNEP, the convexity of the NGNEP guarantees that $\bar{x}^\nu \in X_\nu(\bar{x}^{-\nu})$ and $(x^\nu - \bar{x}^\nu)^\top v_\nu(\bar{x}) \geq 0$ for all $x^\nu \in X_\nu(\bar{x}^{-\nu})$. Summing these inequalities over $\nu \in \NCal$, we obtain the desired QVI. 
\end{proof}
Proposition~\ref{Prop:NGNEP-equivalence} is a generalization of standard results for NEPs~\citep[see, e.g.,][Proposition~2.1]{Mertikopoulos-2019-Learning}, and such a characterization ensures existence results; see~\citet[Theorem~5.2]{Chan-1982-Generalized} or~\citet[Theorem~2]{Harker-1991-Generalized}. Uniqueness results, however, are another matter. While the \textit{diagonally strict concavity} (DSC) condition is sufficient for uniqueness of solutions in NEPs~\citep{Rosen-1965-Existence}, in the GNEP setting it does not even guarantee a connected solution set \citep[see Figure~1 in][and the comments after his Corollary 3.1]{Harker-1991-Generalized}. The condition can be augmented to ensure uniqueness but only under additional conditions that are mainly of theoretical interest~\citep[see][Eq.~(33)]{Harker-1991-Generalized}. In this paper, we do not view uniqueness as a reasonable goal but instead argue that the DSC condition itself---which is also referred to as \emph{strict monotonicity} in convex analysis~\citep{Bauschke-2017-Convex}---is strong enough for global convergence rate estimation for the algorithms. Strict monotonicity and the related notions of \textit{monotonicity} and \textit{strong monotonicity} are useful assumptions for analyzing models and algorithms in the variational inequality literature~\citep{Facchinei-2007-Finite}. 
\begin{definition}
An NGNEP (i.e., $G=(\NCal, \{X_\nu(\cdot)\}_{\nu=1}^N, \{\theta_\nu(\cdot)\}_{\nu=1}^N)$) is said to be
\begin{enumerate}
\item \emph{monotone} if $(x - x')^\top(v(x) - v(x')) \geq 0$ for all $x, x' \in \hat{X}$. 
\item \emph{$\alpha$-strongly monotone} if $(x - x')^\top(v(x) - v(x')) \geq \alpha\|x - x'\|^2$ for all $x, x' \in \hat{X}$.
\end{enumerate} 
\end{definition}
\begin{remark}
Proposition~\ref{Prop:NGNEP-equivalence} can be directly applied to these monotone NGNEPs. Indeed, $(x - x')^\top(v(x) - v(x')) \geq 0$ for all $x, x' \in \hat{X}$ guarantees that $(x^\nu - (x')^\nu)^\top (v_\nu(x^\nu, x^{-\nu}) - v_\nu((x')^\nu, x^{-\nu})) \geq 0$ for all $x^\nu, (x')^\nu \in \hat{X}_\nu$ and all $\nu \in \NCal$. Since $v_\nu(x) = \nabla_\nu \theta_\nu(x)$, we have $\theta_\nu(\cdot, x^{-\nu})$ is convex over $\hat{X}_\nu$ for all $\nu \in \NCal$. Moreover, the equilibrium existence is guaranteed here if $\hat{X}_\nu$ is convex and compact for all $\nu \in \NCal$~\citep[Theorem~2]{Harker-1991-Generalized}. See~\citet{Chan-1982-Generalized, Harker-1991-Generalized, Pang-2005-Quasi} for further results that flow from the monotonicity assumption. 

Strongly monotone NGNEPs are a subclass of monotone NGNEPs that admit a unique solution if $X_\nu(x^{-\nu})$ does not depend on the rival's choices $x^{-\nu}$. This feature is important from an algorithmic viewpoint since a natural quantity for measuring the iteration complexity is the distance between the iterates generated by an algorithm and a unique solution. However, such solution uniqueness is not guaranteed in simplest two-player strongly monotone GNEPs~\citep{Harker-1991-Generalized}, demonstrating that, even with the stringent condition of strong monotonicity, it has been unclear which quantity is suitable for measuring the iteration complexity of the algorithms. 
\end{remark}
\begin{remark}
There have been many algorithms for solving the GNEPs, including relaxation methods~\citep{Uryas-1994-Relaxation, Krawczyk-2000-Relaxation, Von-2009-Relaxation}, penalty methods and augmented Lagrangian methods~\citep{Pang-2005-Quasi, Facchinei-2011-Partial, Fukushima-2011-Restricted, Facchinei-2010-Penalty, Kanzow-2016-Multiplier, Kanzow-2016-Augmented, Kanzow-2018-Augmented}, Newton-type methods~\citep{Facchinei-2009-Generalized, Von-2012-Newton, Dreves-2013-Globalized, Izmailov-2014-Error, Fischer-2016-Globally}, the interior-point potential reduction method~\citep{Dreves-2011-Solution}, Lemke's method~\citep{Schiro-2013-Solution}, and the trust-region method~\citep{Galli-2018-Nonmonotone}. Since NGNEPs are a special class of GNEPs, some of these algorithms can be directly applied to solve monotone NGNEPs and provide a global convergence guarantee. On the other hand, nonasymptotic convergence rates (or iteration complexity bounds) are unknown for these algorithms to the best of our knowledge. 
\end{remark}
We provide a few examples of monotone GNEPs to give a sense of their expressivity. Two examples are linear~\citep{Stein-2016-Cone, Dreves-2016-Solving, Dreves-2017-Computing, Stein-2018-Noncooperative} and the other two are nonlinear. 
\begin{example}[Basic economic market model] Consider a set of firms indexed by $\nu \in \NCal$ that offer the same product in a common market. Each firm acts as a price taker and sells the quantity $x_k^\nu \geq 0$ in the price category $p_k^\nu$, where $k \in \KCal$ for an index set $\KCal$ and where the prices are given. We impose the allocation constraint $\sum_{k \in \KCal} x_k^\nu \leq C^\nu$ for $C^\nu > 0$ and also a public constraint, $\sum_{\nu \in \NCal} x_k^\nu \leq D_k$, which assures that the total offering does not exceed the total demand within each price category. Then, the cost function of firm $\nu$ is given by
\begin{equation*}
\theta_\nu(x^\nu, x^{-\nu}) = c_\nu\sum_{k \in \KCal} x_k^\nu - \sum_{k \in \KCal} p_k^\nu x_k^\nu,  
\end{equation*} 
where $c_\nu \geq 0$ is the marginal cost of firm $\nu$ for producing one unit of product. 
\end{example}
\begin{example}[Extended optimal transportation] Consider a set of competing shippers indexed by $\nu \in \NCal$ who hope to transport the quantity $x_{rt}^\nu$ from manufacturers $r \in \RCal$ to customers indexed by $t \in \TCal$. Manufacturer $r$ has a production capacity, $C_r \geq 0$, and customer $t$ needs at least $D_t \geq 0$ units of this good with $\sum_{r \in \RCal} C_r = \sum_{t \in \TCal} D_t$. This implies the supply constraint $\sum_{\nu \in \NCal}\sum_{t \in \TCal} x_{rt}^\nu = C_r$ and the demand constraint $\sum_{\nu \in \NCal}\sum_{r \in \RCal} x_{rt}^\nu = D_t$. The cost function of firm $\nu$ is then given by
\begin{equation*}
\theta_\nu(x^\nu, x^{-\nu}) = \sum_{r \in \RCal} \sum_{t \in \TCal} c_{rt}^\nu x_{rt}^\nu,  
\end{equation*} 
where $c_{rt}^\nu \geq 0$ is the transportation cost from manufacturer $r$ to consumer $t$ by shipper $\nu$. 
\end{example}
\begin{example}[Cournot competition with capacity constraint] Consider a set of firms indexed by $\nu \in \NCal$, each supplying the market with a quantity $x^\nu \in [0, C_\nu]$ up to the firm's capacity $C_\nu \geq 0$. Suppose that we have $\NCal = \cup_{s=1}^S \NCal_s$ and assume \textit{capacity constraints} $C_s \geq 0$ such that $\sum_{\nu \in \NCal_s} x^\nu \leq C_s$. By the law of supply and demand, the good is priced as a decreasing function of the total amount $\bar{x} = \sum_{\nu \in \NCal} x^\nu$ where the common choice is the linear form of $a-b\bar{x}$ for some $a,b>0$. Then, the cost function of firm $\nu$ is given by 
\begin{equation*}
\theta_\nu(x^\nu, x^{-\nu}) = c_\nu(x^\nu) - x^\nu (a-b\bar{x}),
\end{equation*}
which captures the cost of producing $x^\nu$ units of the good (the function $c_\nu(\cdot)$ represents a marginal cost function of firm $\nu$ and is assumed to be convex) minus the total revenue of such production.      
\end{example}
\begin{example}[Resource allocation auctions with bid capacity] Suppose that there is a service provider with resources $s \in \SCal$. These resources are leased to a set of bidders indexed by $\nu \in \NCal$ who place monetary bids $x_s^\nu$ for the utilization of each resource $s \in \SCal$ up to each player's budget: $\sum_{s \in \SCal} x_s^\nu \leq b^\nu$. There is a \textit{bid capacity} $C_s \geq 0$ for resource $s \in \SCal$ such that $\sum_{\nu \in \NCal} x_s^\nu \leq C_s$. Once all bids are tendered, the unit of resource $s$ allocated to bidder $\nu$ is 
\begin{equation*}
\rho_s^\nu = \frac{q_s x_s^\nu}{d_s + \sum_{\nu \in \NCal} x_s^\nu}, 
\end{equation*}
where $q_s$ denotes the total amount of resource $s$ and $d_s > 0$ is the ``entry barrier" for bidding on it). Then, the cost function of bidder $\nu$ is given by
\begin{equation*}
\theta_\nu(x^\nu, x^{-\nu}) = \sum_{s \in \SCal} (x_s^\nu - c_\nu \rho_\s^\nu), 
\end{equation*}
where $c_\nu \geq 0$ is the marginal gain to bidder $\nu$ from acquiring a unit of resources. 
\end{example}
In addition to the above examples, the class of monotone NGNEPs contains all general convex-concave zero-sum games with linear constraints, monotone NEPs, and all convex potential games with private convex constraints (there exists a convex function $f: \hat{X} \mapsto \br$ such that $v_\nu(x) = \nabla_\nu f(x)$ for all $\nu \in \NCal$). They are also a natural generalization of convex programming problems with linear constraints which constitute the backbone of nonlinear optimization~\citep{Ben-2001-Lectures}. As such, the notion of (strong) monotonicity, which will play a crucial role in the analysis of this paper, is not limiting in practice but encompasses a wide range of application problems arising from economics and online decision making~\citep{Facchinei-2007-Finite}. 

\subsection{Solution concept}
Following up Definition~\ref{def:NGNEP-solution} and Proposition~\ref{Prop:NGNEP-equivalence}, we define the notion of $\epsilon$-solution of monotone NGNEPs. In particular, if $\bar{x}$ is a solution of NGNEPs, then the feasibility condition at $\bar{x}$ holds true:
\begin{equation}\label{opt:feas}
\bar{x} \in \hat{X} \cap \left\{x \in \br^n: A_s x^{\NCal_s} \leq b_s, \ E_s x^{\NCal_s} = d_s, \textnormal{ for all } s \in [S]\right\}, 
\end{equation}
and $(\bar{x} - x)^\top v(\bar{x}) \leq 0$ for all $x \in \hat{X}$ satisfying the following condition: 
\begin{equation}\label{opt:feas-VI}
\textnormal{ for all } \nu \in \NCal \ \left\{\begin{array}{rl}
A_s^\nu x^\nu + \sum_{j \in \NCal_s, j \neq \nu} A_s^j \bar{x}^j - b_s \leq 0, & \forall s \in \ICal_\nu, \\ 
E_s^\nu x^\nu + \sum_{j \in \NCal_s, j \neq \nu} E_s^j \bar{x}^j - d_s = 0, & \forall s \in \ICal_\nu. 
\end{array}\right.
\end{equation}
Inspired by the gap function and the weak solution concepts that are commonly used for iteration complexity analysis of algorithms in the VI literature~\citep{Facchinei-2007-Finite}, we define an $\epsilon$-solution of monotone NGNEPs as follows. 
\begin{definition}[$\epsilon$-solution concept]\label{def:eps-solution}
We say a point $\bar{x} \in \hat{X}$ is an $\epsilon$-solution of monotone NGNEPs if the following $\epsilon$-feasibility condition at $\bar{x}$ holds true:
\begin{equation*}
\|\max\{0, A_s\bar{x}^{\NCal_s} - b_s\}\| \leq \epsilon, \quad \|E_s\bar{x}^{\NCal_s} - d_s\| \leq \epsilon, \quad \textnormal{for all } s \in [S], 
\end{equation*}
and an $\epsilon$-quasi-variational inequality at $\bar{x}$ holds true: for all $\nu \in \NCal$, we have 
\begin{equation*}
(\bar{x}^\nu - x^\nu)^\top v_\nu(x^\nu, \bar{x}^{-\nu}) = \tilde{O}(\epsilon), 
\end{equation*}
for all $x^\nu \in \hat{X}_\nu$ satisfying the following condition: 
\begin{equation*}
\|\max\{0, A_s^\nu x^\nu + \sum_{j \in \NCal_s, j \neq \nu} A_s^j \bar{x}^j - b_s\}\| \leq \epsilon, \quad \|E_s^\nu x^\nu + \sum_{j \in \NCal_s, j \neq \nu} E_s^j \bar{x}^j - d_s\| \leq \epsilon, \quad \textnormal{for all } s \in \ICal_\nu. 
\end{equation*}
\end{definition}
This definition reduces to standard $\epsilon$-optimality notions in special settings. For example,  monotone NGNEPs reduce to monotone VIs when $A_s = 0$, $E_s = 0$, $b_s = 0$ and $d_s = 0$ for all $s \in [S]$. If $\bar{x}$ is a $\epsilon$-weak solution of the VI such that $(\tilde{x} - x)^\top v(x) \leq \epsilon$ for all $x \in \hat{X}$, it is an $\epsilon$-solution. Indeed, letting $x^{-\nu} = \bar{x}^{-\nu}$ for each $\nu \in \NCal$ yields the desired result. Thus, our notion in Definition~\ref{def:eps-solution} is a generalization of the $\epsilon$-weak solution concept, which has been adopted for measuring the iteration complexity in the VI setting~\citep{Nemirovski-2004-Prox, Nesterov-2007-Dual, Malitsky-2015-Projected, Kotsalis-2022-Simple}. Moreover, monotone NGNEPs reduce to linearly constrained convex problems for the case of $N=1$, where the similar notions have been adopted for linearly constrained nonsmooth nonconvex problems~\citep{Jiang-2019-Structured} as well as nonlinearly constrained nonsmooth convex problems~\citep{Rockafellar-1976-Augmented, Yu-2017-Simple, Xu-2021-Iteration}. 
\begin{remark}\label{remark:NGNEP-KKT}
We show that an $\epsilon$-solution is a solution of monotone NGNEPs when $\epsilon=0$. Indeed, we have $\|\max\{0, A_s\bar{x}^{\NCal_s} - b_s\}\| \leq 0$ and $\|E_s\bar{x}^{\NCal_s} - d_s\| \leq 0$ for all $s \in [S]$ and $\bar{x} \in \hat{X}$. Putting these pieces together yields Eq.~\eqref{opt:feas}. Furthermore, we have $(\bar{x} - x)^\top v(x) \leq 0$ for all $x \in \hat{X}$ satisfying the following condition:
\begin{equation}\label{opt:feas-VI-tmp}
\textnormal{ for all } \nu \in \NCal \ \left\{\begin{array}{rl}
A_s^\nu x^\nu + \sum_{j \in \NCal_s, j \neq \nu} A_s^j \bar{x}^j - b_s \leq 0, & \forall s \in \ICal_\nu, \\ 
E_s^\nu x^\nu + \sum_{j \in \NCal_s, j \neq \nu} E_s^j \bar{x}^j - d_s = 0, & \forall s \in \ICal_\nu. 
\end{array}\right.
\end{equation} 
Now it suffices to show that $(\bar{x} - x)^\top v(\bar{x}) \leq 0$ for all $x \in \hat{X}$ satisfying the conditions in Eq.~\eqref{opt:feas-VI-tmp}. Indeed, we fix $x \in \hat{X}_\nu$ such that Eq.~\eqref{opt:feas-VI-tmp} holds at $x$ and let $x(t)=tx + (1-t)\bar{x}$ be a function of $t \in [0, 1]$. Since Eq.~\eqref{opt:feas} holds and $\hat{X}$ is convex, we have $x(t) \in \hat{X}$ and $x(t)$ satisfies Eq.~\eqref{opt:feas-VI-tmp} for any $t \in [0, 1]$. Therefore, we have 
\begin{equation*}
(\bar{x} - x)^\top v(tx + (1-t)\bar{x}) = \tfrac{1}{t}(\bar{x} - x(t))^\top v(x(t)) \leq 0,
\end{equation*}
which implies that (by letting $t \rightarrow 0$) $(\bar{x} - x)^\top v(\bar{x}) \leq 0$. Since $x \in \hat{X}$ is chosen as any point satisfying Eq.~\eqref{opt:feas-VI-tmp}, we obtain the claimed result. 
\end{remark}

An alternative way to characterize the solution concept in the context of GNEPs is based on the introduction of Lagrangian multipliers and Karush-Kuhn-Tucker (KKT) conditions tailored to GNEPs~\citep{Pang-2005-Quasi, Facchinei-2010-Penalty, Kanzow-2016-Multiplier, Bueno-2019-Optimality}. In particular, we say a point $\bar{x} \in \hat{X}$ is a KKT point of NGNEPs when each block $\bar{x}^\nu$ is a KKT point for minimizing the function $\theta_\nu(\cdot, \bar{x}^{-\nu})$ over the set $X_\nu(\bar{x}^{-\nu})$ for all $\nu \in \NCal$; that is, the following feasibility condition at $\bar{x}$ holds true:
\begin{equation*}
\bar{x} \in \hat{X} \cap \left\{x \in \br^n: A_s x^{\NCal_s} \leq b_s, \ E_s x^{\NCal_s} = d_s, \textnormal{ for all } s \in [S]\right\}, 
\end{equation*}
and there exist Lagrangian multipliers $\bar{\lambda}^{\nu, s} \geq 0$ and $\bar{\mu}^{\nu, s}$ such that 
\begin{equation*}
\textnormal{ for all } \nu \in \NCal \ \left\{\begin{array}{rl}
(\bar{x}^\nu - x^\nu)^\top(v_\nu(\bar{x}) + \sum_{s \in \ICal_\nu} ((A_s^\nu)^\top\bar{\lambda}^{\nu, s} + (E_s^\nu)^\top\bar{\mu}^{\nu, s})) \leq 0, & \forall x^\nu \in \hat{X}_\nu, \\
(\bar{\lambda}^{\nu, s})^\top(A_s\bar{x}^{\NCal_s} - b_s) = 0, & \forall s \in \ICal_\nu. 
\end{array}\right.
\end{equation*}
It is natural to investigate the relationship between solutions and KKT points in the context of GNEPs. In general, constraint qualifications (CQs) are necessary for ensuring that a solution is a KKT point~\citep{Bueno-2019-Optimality}. However, monotone NGNEPs have special structure in the sense that $\theta_\nu(\cdot, x^{-\nu})$ is convex for any fixed $x^{-\nu}$, $\hat{X}$ is simple and all the coupled constraints are linear. The features guarantee the equivalence between the solution concept in Definition~\ref{def:NGNEP-solution} and the above KKT conditions tailored to NGNEPs. A proof based on Farkas's lemma and Proposition~\ref{Prop:NGNEP-equivalence} is presented in Appendix~\ref{app:NGNEP-KKT}. 
\begin{theorem}\label{Thm:NGNEP-KKT}
If the NGNEP is monotone, $\bar{x} \in \hat{X}$ is a solution (cf.\ Definition~\ref{def:NGNEP-solution}) if and only if it is a KKT point of this NGNEP. 
\end{theorem}
For completeness, we also present the notion of an $\epsilon$-KKT point. 
\begin{definition}[$\epsilon$-KKT]\label{def:eps-NGNEP-KKT}
We say a point $\bar{x} \in \hat{X}$ is an \emph{$\epsilon$-KKT point} for a monotone NGNEP if the following $\epsilon$-feasibility condition at $\bar{x}$ holds true:
\begin{equation*}
\|\max\{0, A_s\bar{x}^{\NCal_s} - b_s\}\| \leq \epsilon, \quad \|E_s\bar{x}^{\NCal_s} - d_s\| \leq \epsilon, \quad \textnormal{for all } s \in [S], 
\end{equation*}
and there exist some Lagrangian multipliers $\bar{\lambda}^{\nu, s} \geq 0$ and $\bar{\mu}^{\nu, s}$ such that 
\begin{equation*}
\textnormal{ for all } \nu \in \NCal \ \left\{\begin{array}{rl}
(\bar{x}^\nu - x^\nu)^\top(v_\nu(\bar{x}) + \sum_{s \in \ICal_\nu} ((A_s^\nu)^\top\bar{\lambda}^{\nu, s} + (E_s^\nu)^\top\bar{\mu}^{\nu, s})) \leq \epsilon, & \forall x^\nu \in \hat{X}_\nu, \\
\|\min\{ \bar{\lambda}^{\nu, s}, -(A_s\bar{x}^{\NCal_s} - b_s)\}\| \leq \epsilon, & \forall s \in \ICal_\nu. 
\end{array}\right.
\end{equation*}
\end{definition}
\begin{remark}
The notion of $\epsilon$-KKT point has been introduced in~\citet[Definition~5.2]{Bueno-2019-Optimality} for more general GNEPs and Definition~\ref{def:eps-NGNEP-KKT} is that  notion applied to NGNEPs. For the special case of monotone VIs with $\bar{\lambda}^{\nu, s} = 0$ and $\bar{\mu}^{\nu, s} = 0$, Definition~\ref{def:eps-NGNEP-KKT} corresponds to the $\epsilon$-strong solution concept which is different from the $\epsilon$-weak solution concept from an algorithmic point of view (see~\citet{Diakonikolas-2020-Halpern} for recent progress).  Can we design some variants of the algorithms in this paper to pursue an $\epsilon$-KKT point as in Definition~\ref{def:eps-NGNEP-KKT} and prove their global convergence rate? This is an interesting open question.
\end{remark}

\subsection{Accelerated mirror-prox scheme}
We review the deterministic accelerated mirror-prox (AMP) algorithm~\citep{Chen-2017-Accelerated}, which provides a multi-step acceleration scheme for solving \textit{composite} smooth and monotone VIs with an optimal complexity bound guarantee. More specifically, the AMP algorithm aims at solving the following class of VIs: 
\begin{equation}\label{subprob:CMVI}
\textnormal{Find } z^\star \in Z: \quad (z - z^\star)^\top(F(z^\star) + \nabla G(z^\star)) \geq 0, \quad \textnormal{for all } z \in Z,  
\end{equation} 
where $F: Z \mapsto \br^n$ is $\ell_F$-Lipschitz and monotone, $G: X \mapsto \br$ is $\ell_G$-smooth and convex, and and $Z$ is a closed and convex nonempty set with diameter $D_Z>0$. With the initialization $z_1 = z_1^{\textnormal{ag}} = w_1 \in Z$, a typical iteration of the AMP algorithm with Euclidean projection is 
\begin{equation*}
\begin{array}{lcl}
z_k^{\textnormal{md}} & = & (1-\alpha_k)z_k^{\textnormal{ag}} + \alpha_k w_k, \\
z_{k+1} & = & \argmin_{z \in Z} \gamma_k(z - w_k)^\top(F(w_k) + \nabla G(z_k^{\textnormal{md}})) + \tfrac{1}{2}\|z - w_k\|^2, \\
w_{k+1} & = & \argmin_{z \in X} \gamma_k(z - w_k)^\top(F(z_{k+1}) + \nabla G(z_k^{\textnormal{md}})) + \tfrac{1}{2}\|z - w_k\|^2, \\
z_{k+1}^{\textnormal{ag}} & = & (1-\alpha_k)z_k^{\textnormal{ag}} + \alpha_k z_{k+1}.  
\end{array}
\end{equation*} 
An iteration complexity bound of the AMP algorithm in terms of the number of gradient evaluations is presented in~\citet[Corollary~1]{Chen-2017-Accelerated}.  We summarize the result in the following theorem.\footnote{We only focus on the deterministic algorithm with Euclidean projection. As such, we provide a simplified version of~\citet[Corollary~1]{Chen-2017-Accelerated} in the Euclidean setting and $\sigma=0$ in Theorem~\ref{Thm:AMP-MT}.}
\begin{theorem}\label{Thm:AMP-MT}
Suppose that $F: Z \mapsto \br^n$ is $\ell_F$-Lipschitz and monotone and $G: Z \mapsto \br$ is $\ell_G$-smooth and convex, and we set the parameters $\alpha_k=\tfrac{2}{k+1}$ and $\gamma_k=\tfrac{k}{4\ell_G+3k\ell_F}$. Then, the following inequality holds for all $k \geq 2$, 
\begin{equation*}
\max_{z \in Z} \left\{G(z_k^{\textnormal{ag}}) - G(z) + (z_k^{\textnormal{ag}} - z)^\top F(z)\right\} \leq \tfrac{16\ell_G D_Z^2}{k(k-1)} + \tfrac{12\ell_F D_Z^2}{k-1},
\end{equation*}
and the required number of gradient evaluations to return the point $\hat{z} \in Z$ satisfying that $\max_{z \in Z} \{G(\hat{z}) - G(z) + (\hat{z} - z)^\top F(z)\} \leq \epsilon$ is bounded by $O\left(\sqrt{\frac{\ell_G D_Z^2}{\epsilon}} + \frac{\ell_F D_Z^2}{\epsilon}\right)$. 
\end{theorem}
\begin{remark}
There has been considerable interest in the development of algorithms for solving monotone VI problems; among them we highlight the extragradient method~\citep{Korpelevich-1976-Extragradient, Tseng-2008-Accelerated}, mirror-prox~\citep{Nemirovski-2004-Prox}, dual extrapolation~\citep{Nesterov-2007-Dual}, reflected gradient~\citep{Malitsky-2015-Projected}, accelerated mirror-prox~\citep{Chen-2017-Accelerated}, Halpern iteration~\citep{Diakonikolas-2020-Halpern} and operator extrapolation~\citep{Kotsalis-2022-Simple}. For a general introduction, see~\citet{Facchinei-2007-Finite}.
\end{remark}
\begin{remark}
The subproblems in our frameworks are in the form of Eq.~\eqref{subprob:CMVI} with $Z = \hat{X}$, $\ell_G = O(\tfrac{1}{\epsilon})$ for some $\epsilon>0$ and $\ell_F=O(1)$. As such, the AMP algorithm is more suitable than other algorithms for playing a role in the subroutine for solving monotone NGNEPs in the sense that it achieves an complexity bound with better dependence on $1/\epsilon$.
\end{remark}
We provide the iteration complexity bound of the AMP algorithm for solving the VI in Eq.~\eqref{subprob:CMVI} in terms of the number of gradient evaluations in which $F: Z \mapsto \br^n$ is further assumed to be $\alpha$-strongly monotone. We summarize the result in Theorem~\ref{Thm:AMP-SMT} and refer interested readers to Appendix~\ref{app:AMP-proof} for the proof details.  
\begin{theorem}\label{Thm:AMP-SMT}
Suppose that $F: X \mapsto \br^n$ is $\ell_F$-Lipschitz and $\alpha$-strongly monotone and $G: Z \mapsto \br$ is $\ell_G$-smooth and convex, and we set the parameters $\alpha_k = \tfrac{1}{4}\min\{\frac{\alpha}{\ell_F}, \sqrt{\frac{\alpha}{\ell_G}}\}$ and $\gamma_k = \tfrac{\alpha_k}{\alpha}$. Then, the following inequality holds for all $k \geq 2$:
\begin{equation*}
\max_{z \in Z} \left\{G(z_k^{\textnormal{ag}}) - G(z) + (z_k^{\textnormal{ag}} - z)^\top F(z)\right\} \leq \left(1-\tfrac{1}{4}\min\left\{\tfrac{\alpha}{\ell_F}, \sqrt{\tfrac{\alpha}{\ell_G}}\right\}\right)^{k-1}\left(\ell_F+\tfrac{\ell_G+\alpha}{2}\right)D_Z^2,
\end{equation*}
and the required number of gradient evaluations to return the point $\hat{z} \in Z$ satisfying that $\max_{z \in Z} \{G(\hat{z}) - G(z) + (\hat{z} - z)^\top F(z)\} \leq \epsilon$ is bounded by $O\left(\left(\sqrt{\frac{\ell_G}{\alpha}} + \frac{\ell_F}{\alpha}\right)\log\left(\frac{(\ell_F+\ell_G)D_Z^2}{\epsilon}\right)\right)$. 
\end{theorem}
\begin{remark}
Theorem~\ref{Thm:AMP-SMT} complements the results in~\citet{Chen-2017-Accelerated} and provides a more complete picture for the optimal iteration complexity bounds of deterministic first-order algorithms for composite smooth and (strongly) monotone VIs in the form of Eq.~\eqref{subprob:CMVI}. An open problem is to consider extensions to the stochastic setting. 
\end{remark}

\section{Main Results}\label{sec:results}
We propose two first-order algorithms---one based on quadratic penalty method (QPM) and the other based on augmented Lagrangian method (ALM)---where each iteration consists in a subroutine for approximately solving a smooth and monotone penalized VI using the AMP algorithm. We carry out an estimation of a global convergence rate for these algorithms in monotone and strongly monotone settings and present iteration complexity bounds in terms of the number of gradient evaluations. Our algorithms update penalty parameters adaptively and thus are favorable from a practical viewpoint.  

\subsection{Quadratic penalty method}
Our first framework is a natural combination of the quadratic penalty method and the AMP algorithm.  We refer to it as the \textit{accelerated mirror-prox quadratic penalty method} (AMPQP). The idea is to transform an NGNEP into a sequence of structured quadratic penalized NEPs and approximately solve each one in the inner loop using a subroutine based on the AMP algorithm. Indeed, we consider a quadratic penalization of the NGNEP where each player's minimization problem is given by  
\begin{equation}\label{prob:NGNEP-QP}
\min_{x^\nu \in \hat{X}^\nu} \theta_\nu(x^\nu, x^{-\nu}) +  \left(\sum_{s \in \ICal_\nu} \tfrac{\beta^s}{2}\|\max\{0, A_s x^{\NCal_s}-b_s\}\|^2 + \tfrac{\rho^s}{2}\|E_s x^{\NCal_s}-d_s\|^2\right), 
\end{equation}
and where $(\beta^s, \rho^s) \in \br_+ \times \br_+$ stand for penalty parameters associated with inequality and equality constraints. Since the function $\theta_\nu(\cdot, x^{-\nu})$ is convex with Lipschitz gradient and the constraint set $\hat{X}^\nu$ does not depend on the rival's choices, the quadratic penalization of the NGNEP in Eq.~\eqref{prob:NGNEP-QP} is an NEP. 

For simplicity, we define the functions $g_1: \br^n \mapsto \br$ and $h_1: \br^n \mapsto \br$ by
\begin{equation*}
g_1(x) = \sum_{s=1}^S \tfrac{\beta^s}{2}\|\max\{0, A_s x^{\NCal_s} - b_s\}\|^2, \quad h_1(x) = \sum_{s=1}^S \tfrac{\rho^s}{2}\|E_s x^{\NCal_s}-d_s\|^2. 
\end{equation*}
By concatenating the first-order optimality conditions of Eq.~\eqref{prob:NGNEP-QP} for each player, we aim at solving the following smooth and monotone penalized VI: 
\begin{equation}\label{prob:PVI-QP}
\textnormal{Find } x \in \hat{X}: \quad (x' - x)^\top(v(x) + \nabla g_1(x) + \nabla h_1(x)) \geq 0, \textnormal{ for all } x' \in \hat{X}.     
\end{equation}
It is worth mentioning that this VI is in the form of Eq.~\eqref{subprob:CMVI} and can be approximately solved by the AMP algorithm with the theoretical guarantees that we have explicated. 

The introduction of quadratic penalization based on squared Euclidean norm was a milestone in optimization~\citep{Bertsekas-1976-Penalty, Aybat-2011-First, Lan-2013-Iteration, Necoara-2019-Complexity}. It has been, however, less explored for GNEPs possibly because the quadratic penalty method cannot recover an exact solution for a finite penalty parameter~\citep[see][Example~4(b)]{Ba-2022-Exact}. Nevertheless, if our goal is finding an $\epsilon$-approximate solution and estimating the iteration complexity bounds in terms of the number of gradient evaluations, the quadratic penalty method has a clear advantage over other exact penalty methods; indeed, each subproblem in exact penalty methods is nondifferentiable and problematic from the numerical viewpoint~\citep{Facchinei-2010-Penalty}. Although some algorithms for NEPs work without differentiability, e.g., the relaxation method~\citep{Uryas-1994-Relaxation, Von-2009-Relaxation} or the proximal-like method~\citep{Flaam-1996-Equilibrium}, they require solving certain constrained optimization problems which are only tractable for sufficiently smooth data. Another approach is to use nonlinear system solvers after smoothing the nondifferentiable terms~\citep[Section~3]{Facchinei-2010-Penalty}. To the best of our knowledge, the global convergence rates of these approaches are unknown and the nondifferentiability makes the iteration complexity analysis difficult. In contrast, the quadratic penalization in Eq.~\eqref{prob:PVI-QP} is smooth and monotone such that a standard iteration complexity analysis can be applied.

To further enhance the practical viability of the algorithm, there are two issues that need to be addressed: (i) when to stop the AMP algorithm for each inner loop; and (ii) how to update the penalty parameters $(\beta^s, \rho^s)$. In principle, there are adaptive strategies adopted in the exact penalty methods~\citep{Facchinei-2011-Partial, Facchinei-2010-Penalty, Fukushima-2011-Restricted} such that  global convergence guarantees can be achieved given a convergent algorithm for each nondifferentiable NEP. However, all of these theoretical results are asymptotic and the global convergence rates of these adaptive approaches are unknown. 
\begin{algorithm}[!t]
\caption{Accelerated Mirror-Prox Quadratic Penalty Method (AMPQP)}\label{alg:AMPQP}
\begin{algorithmic}
\STATE \textbf{Input:} $\gamma > 1$, $\delta_0 \in (0, 1)$ and $x_0 \in \hat{X}$. 
\STATE \textbf{Initialization:} $(\beta_0^s, \rho_0^s) \in \br_+ \times \br_+$ for all $s \in [S]$. 
\FOR{$k=0, 1, 2, \ldots, T-1$}
\STATE Update $(\beta_{k+1}^s, \rho_{k+1}^s) = (\gamma\beta_k^s, \gamma\rho_k^s)$ for all $s \in [S]$ and $\delta_{k+1} = \tfrac{\delta_k}{\gamma}$; 
\STATE Compute $x_{k+1} = \textsc{amp}((\beta_{k+1}^s, \rho_{k+1}^s)_{s \in [S]}, \delta_{k+1}, x_k)$. 
\ENDFOR
\STATE \textbf{Output:} $x_T$. 
\end{algorithmic}
\end{algorithm}

We summarize the pseudocode of the AMPQP algorithm in Algorithm~\ref{alg:AMPQP}. In particular, we follow an adaptive strategy akin to those of~\citet{Facchinei-2010-Penalty, Facchinei-2011-Partial, Fukushima-2011-Restricted}, who suggest updating $(\beta_{k+1}^s, \rho_{k+1}^s) = (\gamma\beta_k^s, \gamma\rho_k^s)$ with $\gamma > 1$ at the $k^\textnormal{th}$ iteration. Also, $\delta_{k+1}$ is updated accordingly and we let the AMP subroutine return a $\delta_{k+1}$-approximate solution $\hat{x} \in \hat{X}$ to the VI in Eq.~\eqref{prob:PVI-QP} with $(\beta^s, \rho^s) = (\beta_{k+1}^s, \rho_{k+1}^s)$ as follows: 
\begin{equation}\label{condition:delta-approx}
\max_{x \in \hat{X}} \left\{g_1(\hat{x}) + h_1(\hat{x}) - g_1(x) - h_1(x) + (\hat{x} - x)^\top v(x)\right\} \leq \delta_{k+1}. 
\end{equation}
Our subroutine implements the AMP algorithm with an input $x_k$ from previous iteration and returns $x_{k+1}$ which is a $\delta_{k+1}$-approximate solution. For simplicity, we write the subroutine in the compact form as $x_{k+1} = \textsc{amp}((\beta_{k+1}^s, \rho_{k+1}^s)_{s \in [S]}, \delta_{k+1}, x_k)$.

Intuitively, the idea of Algorithm~\ref{alg:AMPQP} can be described as follows. Suppose that we know the ``correct" values of the penalty parameters $\{\beta^s, \rho^s\}_{s \in [S]}$ and threshold $\delta$ such that a $\delta$-approximate solution of the penalized VI with these parameters is an $\epsilon$-solution of NGNEPs. Then, a single application of the AMP algorithm would give us an $\epsilon$-solution. However, the correct values are unknown beforehand and thus we need to design the adaptive strategies in Algorithm~\ref{alg:AMPQP} to handle this issue. Even though we stated for simplicity that the subroutine $\textsc{amp}(\cdot)$ should produce $x_{k+1}$ based only on $x_k$, it is plausible to use any information gathered in previous iterations as well. Nevertheless, after the penalty parameters are updated, we have changed the subproblem and all old information should be discarded (since it is related to a different problem). As such, we prefer to write $x_{k+1} = \textsc{amp}((\beta_{k+1}^s, \rho_{k+1}^s)_{s \in [S]}, \delta_{k+1}, x_k)$, where $x_k$ is regarded as a starting point for the application of the AMP algorithm. 

Notably, it is intractable to verify if $x_T$ is an $\epsilon$-solution using Definition~\ref{def:eps-solution}. In practice, an alternative approach is to set the appropriate stopping criteria. For example, we can run the algorithm until either of the following conditions hold true: (i) a theoretically derived maximum number of iterations is reached, (ii) the iterative gap $\|x_T - x_{T-1}\|$ is sufficiently small, or (iii) the minimum of penalty parameters, $\min_{1 \leq s \leq S} \{\beta_T^s, \rho_T^s\}$, is sufficiently large. The same practical issues hold for verifying if $x_{k+1}$ is an $\delta_{k+1}$-solution of the VI subproblem using Eq.~\eqref{condition:delta-approx}. In our implementation, we set the stopping criteria for the AMP subroutine if either a theoretically derived maximum number of iterations (cf.\ Theorem~\ref{Thm:AMP-MT} and~\ref{Thm:AMP-SMT}) is reached or the iterative gap is sufficiently small. 

\subsection{Augmented Lagrangian method}
Our second framework combines the inexact augmented Lagrangian method with the AMP algorithm. We refer to it as the \textit{accelerated mirror-prox augmented Lagrangian method} (AMPAL). The idea is again to transform an NGNEP into a sequence of quadratic penalized NEPs but using the augmented Lagrangian function and to approximately solve the VI-type subproblem in the inner loop using the AMP algorithm. Indeed, we define the following augmented Lagrangian function for each player's minimization problem: 
\begin{equation*}
\LCal_\nu(x, \lambda, \mu) = \theta_\nu(x^\nu, x^{-\nu}) + \left(\sum_{s \in \ICal_\nu} \tfrac{\beta^s}{2}\|\max\{0, A_sx^{\NCal_s}-b_s+\tfrac{\lambda^s}{\beta^s}\}\|^2 + \tfrac{\rho^s}{2}\|E_sx^{\NCal_s}-d_s+\tfrac{\mu^s}{\rho^s}\|^2\right),  
\end{equation*}
where $(\beta^s, \rho^s) \in \br_+ \times \br_+$ and $(\lambda^s, \mu^s) \in \br^{m_s} \times \br^{e_s}$ stand for penalty parameters and Lagrangian multipliers associated with inequality and equality constraints. The augmented Lagrangian function has some important advantages: (i) it is convex in the variable $x^\nu$ and concave in the variable $\{(\lambda^s, \mu^s)\}_{s \in \ICal_\nu}$; (ii) its gradient with respect to $x^\nu$ is Lipschitz and the constraint set $\hat{X}^\nu$ does not depend on the rival's choices. Putting these pieces together implies that a joint minimization of the augmented Lagrangian function for all players forms an NEP. For simplicity, we define the functions $g_2: \br^n \mapsto \br$ and $h_2: \br^n \mapsto \br$ by
\begin{equation*}
g_2(x) = \sum_{s=1}^S \tfrac{\beta^s}{2}\|\max\{0, A_s x^{\NCal_s} - b_s + \tfrac{\lambda^s}{\beta^s}\}\|^2, \quad h_2(x) = \sum_{s=1}^S \tfrac{\rho^s}{2}\|E_s x^{\NCal_s} - d_s + \tfrac{\mu^s}{\rho^s}\|^2.
\end{equation*}
By concatenating the first-order optimality conditions of minimizing the augmented Lagrangian function for each player, we aim at solving the following smooth and monotone penalized VI: 
\begin{equation}\label{prob:PVI-AL}
\textnormal{Find } x \in \hat{X}: \quad (x' - x)^\top(v(x) + \nabla g_2(x) + \nabla h_2(x)) \geq 0, \textnormal{ for all } x' \in \hat{X}.     
\end{equation}
This VI is in the form of Eq.~\eqref{subprob:CMVI} and can be approximately solved by the AMP algorithm with the theoretical guarantees that we have explicated. 

The augmented Lagrangian method was proposed by~\citet{Hestenes-1969-Multiplier} and~\citet{Powell-1969-Method} and its convergence has been studied extensively in the optimization literature.  Milestones include early results on asymptotic convergence and local convergence~\citep{Rockafellar-1973-Dual, Rockafellar-1973-Multiplier, Rockafellar-1976-Augmented} and recent results on the nonasymptotic global convergence rate~\citep{Lan-2016-Iteration, Xu-2021-Iteration}. Each iteration of the algorithm consists in first minimizing the augmented Lagrangian function with respect to primal variables while fixing the dual variables and then performing a dual gradient ascent update. It is, however, generally intractable to exactly minimize the augmented Lagrangian function with respect to primal variables, leading to the inexact variant~\citep{Xu-2021-Iteration} where each subproblem is solved within a desired tolerance. This idea was later extended to GNEPs and QVIs~\citep{Kanzow-2016-Augmented, Kanzow-2018-Augmented, Bueno-2019-Optimality} and the resulting algorithms were shown to outperform the penalty-type methods.

As pointed out by~\citet{Kanzow-2016-Augmented}, the practical advantages come from the smoothness of each subproblem and the safeguarding effect of multipliers. In this context, one often solves each subproblem by appealing to well-known second-order methods, e.g., semismooth Newton methods~\citep{Kummer-1988-Newton, Qi-1993-Nonsmooth} and Levenberg-Marquardt methods~\citep{Levenberg-1944-Method, Marquardt-1963-Algorithm, Yamashita-2001-Rate}. However, a singularity arises in GNEPs when some players share the same constraints~\citep{Facchinei-2009-Generalized} and this makes the iteration complexity analysis of these approaches difficult.
\begin{algorithm}[!t]
\caption{Accelerated Mirror-Prox Augmented Lagrangian Method (AMPAL)}\label{alg:AMPAL}
\begin{algorithmic}
\STATE \textbf{Input:} $\gamma > 1$, $\delta_0 \in (0, 1)$ and $x_0 \in \hat{X}$. 
\STATE \textbf{Initialization:} $(\beta_0^s, \rho_0^s) \in \br_+ \times \br_+$ and $(\lambda_0^s, \mu_0^s) \in \br^{m_s} \times \br^{e_s}$ for all $s \in [S]$. 
\FOR{$k=0, 1, 2, \ldots, T-1$}
\STATE Update $(\beta_{k+1}^s, \rho_{k+1}^s) = (\gamma\beta_k^s, \gamma\rho_k^s)$ for all $s \in [S]$ and $\delta_{k+1} = \tfrac{\delta_k}{\gamma}$; 
\STATE Compute $x_{k+1} = \textsc{amp}((\beta_{k+1}^s, \rho_{k+1}^s)_{s \in [S]}, (\lambda_k^s, \mu_k^s)_{s \in [S]}, \delta_{k+1}, x_k)$; 
\STATE Update $\lambda_{k+1}^s$ and $\mu_{k+1}^s$ for all $s \in [S]$ by 
\begin{equation*}
\lambda_{k+1}^s = \max\{0, \lambda_k^s + \beta_{k+1}^s(A_s x_{k+1}^{\NCal_s} - b_s)\}, \quad \mu_{k+1}^s = \mu_k^s + \rho_{k+1}^s(E_s x_{k+1}^{\NCal_s} - d_s).
\end{equation*}
\ENDFOR
\STATE \textbf{Output:} $\hat{x}_T = \tfrac{1}{T}\sum_{k=1}^T x_k$. 
\end{algorithmic}
\end{algorithm}

We summarize the pseudocode of the AMPAL algorithm in Algorithm~\ref{alg:AMPAL}. In particular, there are several heuristics for adapting penalty parameters in the augmented Lagrangian methods~\citep{Kanzow-2016-Augmented, Kanzow-2018-Augmented, Bueno-2019-Optimality} and some of them have been proven very efficient in practice~\citep[see][Section~6]{Kanzow-2016-Augmented}. We use these strategies in the AMPAL algorithm together with updating $(\beta_{k+1}^s, \rho_{k+1}^s) = (\gamma\beta_k^s, \gamma\rho_k^s)$. Also, $\delta_{k+1}$ is updated accordingly and we let the AMP subroutine return a $\delta_{k+1}$-approximate solution $\hat{x} \in \hat{X}$ to the VI in Eq.~\eqref{prob:PVI-AL} with $(\beta^s, \rho^s) = (\beta_{k+1}^s, \rho_{k+1}^s)$:
\begin{equation*}
\max_{x \in \hat{X}} \ \{g_2(\hat{x}) + h_2(\hat{x}) - g_2(x) - h_2(x) + (\hat{x} - x)^\top v(x)\} \leq \delta_{k+1}. 
\end{equation*}
Similar to the AMPQP algorithm, our subroutine implements the AMP algorithm with an input $x_k$ from the previous iteration and returns $x_{k+1}$ which is a $\delta_{k+1}$-approximate solution. We write the subroutine as $x_{k+1} = \textsc{amp}((\beta_{k+1}^s, \rho_{k+1}^s)_{s \in [S]}, (\lambda_k^s, \mu_k^s)_{s \in [S]}, \delta_{k+1}, x_k)$ and update the Lagrangian multipliers $(\lambda_{k+1}^s, \mu_{k+1}^s)_{s \in [S]}$ using $x_{k+1}$ and $(\beta_{k+1}^s, \rho_{k+1}^s)_{s \in [S]}$. In practice, we set the appropriate stopping criteria rather than directly verifying if $x_T$ is an $\epsilon$-solution using Definition~\ref{def:eps-solution} and if $x_{k+1}$ is a $\delta_{k+1}$-approximate solution using Eq.~\eqref{condition:delta-approx}. 

\subsection{Iteration complexity bound}
Our first result is the iteration complexity bound for Algorithm~\ref{alg:AMPQP} when applied to  monotone and strongly monotone NGNEPs.
\begin{theorem}\label{Thm:AMPQP}
For a given tolerance $\epsilon > 0$, the required number of gradient evaluations for Algorithm~\ref{alg:AMPQP} to return an $\epsilon$-solution (cf.\ Definition~\ref{def:eps-solution}) is upper bounded by
\begin{equation*}
N_{\textnormal{grad}} = \left\{\begin{array}{ll}
O(\epsilon^{-1}), & \textnormal{if } \alpha = 0, \\
O(\epsilon^{-1/2}\log(1/\epsilon)), & \textnormal{if } \alpha > 0, \\
\end{array}\right. 
\end{equation*}
where the two lines refer to the case of monotone and strongly monotone NGNEPs. 
\end{theorem}
For both monotone and strongly monotone NGNEPs, our iteration complexity bounds are new and generalize classical results. Indeed, a monotone NGNEP reduces to a linearly constrained convex problem when $N=1$ and our algorithm achieves the lower bound for any deterministic first-order algorithms; see~\citet[Theorem~4.1]{Xu-2021-Iteration}. A monotone NGNEP also reduces to a monotone VI when $A_s=0$, $E_s=0$, $b_s=0$ and $d_s=0$ for all $s \in [S]$ and our algorithm also achieves the lower bound for any deterministic first-order algorithms; see~\citet[Lemma~16]{Diakonikolas-2020-Halpern}. Moreover, a strongly monotone NGNEP reduces to a linearly constrained strongly convex problem when $N=1$ and the lower bound for any deterministic first-order algorithm is $\Omega(\epsilon^{-1/2})$~\citep[Theorem~4.2]{Ouyang-2021-Lower} which is attained by Algorithm~\ref{alg:AMPQP} up to log factors. Although existing lower bounds for linearly constrained convex problems and monotone VIs are derived using \textit{slightly different} notions, the matching orders demonstrate the theoretical efficiency of Algorithm~\ref{alg:AMPQP}. 

Our second result is the iteration complexity bound for Algorithm~\ref{alg:AMPAL} when applied to solve monotone and strongly monotone NGNEPs. 
\begin{theorem}\label{Thm:AMPAL}
For a given tolerance $\epsilon > 0$, the required number of gradient evaluations for Algorithm~\ref{alg:AMPAL} to return an $\epsilon$-solution (cf.\ Definition~\ref{def:eps-solution}) is upper bounded by
\begin{equation*}
N_{\textnormal{grad}} = \left\{
\begin{array}{ll}
O(\epsilon^{-1}\log(1/\epsilon)), & \textnormal{if } \alpha = 0, \\
O(\epsilon^{-1/2}\log(1/\epsilon)), & \textnormal{if } \alpha > 0, \\
\end{array}
\right. 
\end{equation*}
where the two lines refer to the case of monotone and strongly monotone NGNEPs. 
\end{theorem}
Theorem~\ref{Thm:AMPAL} shows that Algorithm~\ref{alg:AMPAL} recovers the optimal iteration complexity bound of the inexact augmented Lagrangian methods for solving linearly constrained (strongly) convex problems~\citep[c.f.][]{Lan-2016-Iteration, Xu-2021-Iteration}. Our results show that Algorithm~\ref{alg:AMPAL} matches the lower bound for any deterministic first-order algorithm up to log factors. 

\section{Experiment}\label{sec:experiments}
We conduct experiments using several datasets from~\citet{Facchinei-2009-Penalty}. All of the algorithms were implemented with MATLAB R2020b on a MacBook Pro with an Intel Core i9 2.4GHz (8 cores and 16 threads) and 16GB memory.

Although some of the GNEPs in~\citet{Facchinei-2009-Penalty} (i.e., \textbf{A1}-\textbf{A9}) are more general than the NGNEPs considered in Definition~\ref{def:NGNEPs}, we can implement our Algorithm~\ref{alg:AMPQP} and~\ref{alg:AMPAL} with player-specific parameters: $u_{\max}^\nu = 10^6$ and $\beta_0^\nu = \rho_0^\nu = 1$ for each $\nu \in \NCal$, where the AMP algorithm is used for solving a general VI.\footnote{We set $G=0$ in which case the AMP algorithm reduces to the extragradient algorithm.} For other GNEPs (i.e., \textbf{A11}-\textbf{A18}), we use $u_{\max}^s = 10^6$ and $\beta_0^s = \rho_0^s = 1$ for every $s \in [S]$. Following the strategies used in~\citet[Algorithm~11 (S4)]{Kanzow-2016-Augmented}, we do not update $(\beta, \rho)$ at each iteration but only when the certain feasibility conditions are satisfied. The associated parameters are chosen according to the size of the problem: $\gamma=4$ if $n < 100$ and $\gamma=2$ otherwise. Compared to~\citet{Kanzow-2016-Augmented}, this represents a less aggressive penalization and is thus suitable for our algorithm. Indeed, for the subproblem solving,~\citet{Kanzow-2016-Augmented} employed a Levenberg-Marquardt algorithm~\citep{Fan-2005-Quadratic} that used both first-order and second-order information while our algorithm only exploits first-order information. Moreover, we initialize the multipliers $(\lambda_0, \mu_0)$ using the same nonnegative least-squares approach as in~\citet{Kanzow-2016-Augmented} and also set the same stopping criterion but with the tolerance $10^{-4}$. The maximum iteration number is set as $50$ and the maximum penalty parameter is set as $10^{12}$. 

Since we perform a full penalization, the subproblems are unconstrained NEPs. Solving these subproblems is equivalent to computing a solution of the nonlinear equation $F(\x) + \nabla G(\x) =0$. We use the AMP algorithm for subproblem solving and set $\ell_G$ and $\ell_F$ properly depending on the problem. The maximum iteration number is set as $2000$ and we stop each inner loop when either the iteration number exceeds $2000$ or $\|F(\x) + \nabla G(\x)\|$ is below $10^{-6}$. 
\begin{table}[!t]
\centering \caption{Numerical results for Algorithm~\ref{alg:AMPAL}. $N$ denotes the number of players, $n$ is the total number of variables, $k$ is the number of outer iterations, $i_{\textnormal{total}}$ is the total  number of inner iterations and F denotes a failure. We also include $R_f$, $R_o$ and $R_c$ which measure the feasibility, optimality and complementary slackness at the solution in terms of KKT condition; see~\citet{Kanzow-2016-Augmented} for the definition.}\label{tab:AMPAL}\small
\begin{tabular}{|cccc|cccccc|} \hline
Example & $N$ & $n$ & $x_0$ & $k$ & $i_{\textnormal{total}}$ & $R_f$ & $R_o$ & $R_c$ & $\rho_{\max}$ \\ \hline
A.1   		& 10 	& 10		& 0.01 	&   7 &   	393 		& 	4.7e-05 	& 1.9e-06 		& 1.3e-05 		&  4 			\\
			&    	&     		&  0.1  	&   7 &		378 		& 	2.9e-05 	& 1.9e-06 		& 7.9e-06 		&  4 			\\
      		&    	&     		&    1  	&   9 &  		7296 	&	4.0e-08 	& 1.3e-06 		& 1.1e-08 		&  4096 	\\ \hline
A.2   		& 10 	& 10 		& 0.01 	&   7 &  		3725 	& 	2.9e-05 	& 1.0e-06 		& 8.7e-06 	&  16 		\\
      		&    	&     		&  0.1 	&   6 &  	3562 	& 	9.6e-05 	& 2.8e-06 	& 2.9e-05 	&  16 		\\
      		&    	&     		&    1 	& 	 F 	& 					& 					& 					& 					&  	 		\\ \hline
A.3   		&  3 	& 7 		&    0 	& 	 F 	& 					& 					& 					& 					&  	 		\\
      		&    	&     		&    1 	& 	 F 	& 					& 					& 					& 					& 				\\
      		&    	&     		&   10 	&  12 &  	5837 	& 	4.3e-07 	& 9.9e-07 		& 3.6e-05 	&  1024 	\\ \hline
A.4   		&  3 	& 7 		&    0 	&  12 &   	618 		&  2.0e-07 	& 8.3e-07 	& 6.1e-05 		&  1024 	\\
      		&   	&     		&    1 	&   0 	&     	0 			&  0 				& 0 				& 0 				&  1 			\\
      		&    	&     		&   10 	&  14	&  	2395 	& 	8.2e-08 	& 9.3e-07 		& 1.5e-05 		&  4096 	\\ \hline
A.5   		&  3 	& 7 		&    0 	&   8 &  		2830 	& 	2.3e-05 	& 9.9e-07 		& 5.2e-05 	&  64 		\\
      		&    	&     		&    1 	&   9 &  		3289 	&  9.0e-06 	& 9.9e-07 		& 2.2e-05 	&  64 		\\
      		&   	&     		&   10 	&   9 &  		5327 	& 	3.7e-05 	& 1.1e-06 		& 9.0e-05 	&  256 		\\ \hline
A.6   		&  3 	& 7 		&    0 	&  12	&  	8636 	& 	1.0e-07 	& 1.0e-06 		& 2.1e-05 		&  1024 	\\
      		&    	&     		&    1 	&  12	&  	7578 	& 	2.1e-08 	& 1.2e-06 		& 2.4e-05 	&  1024 	\\
      		&    	&     		&   10 	& 	F 	& 					& 					& 					& 					&  			\\ \hline
A.7   		&  4 	& 20 	&    0 	&  14 & 		22844 	& 	1.2e-07 	& 7.2e-05 		& 5.3e-05 	&  4096 	\\
      		&    	&     		&    1 	&  13 & 		22609 	& 	4.4e-08 	& 4.9e-05 	& 2.0e-05 	&  4096 	\\
      		&    	&     		&   10 	&  17 & 		23796 	&  2.0e-07 	& 9.3e-07 		& 7.8e-05 		&  4096 	\\ \hline
A.8   		&  3 	& 3 		&    0 	& 	F	& 					& 					& 					& 					& 		 		\\
      		&    	&     		&    1 	&  1 	&   	115 		& 	1.7e-06 	& 4.1e-06 	& 1.7e-06 		&  1 			\\
      		&    	&     		&   10 	&  3 	&   	267 		&  0 				& 8.5e-07 	& 2.1e-07 		&  4 			\\ \hline
A.9a 		&  7 	& 56 	&    0 	&  9 	&  	2353	 	& 	8.2e-06 	& 1.0e-06 		& 2.2e-05 	&  4 			\\
A.9b 	&  7 	& 112 	&    0 	& 16 	&  	1793 	& 	4.2e-06 	& 8.2e-07 		& 4.1e-05 		&  1 			\\ \hline
A.11  		&  2 	& 2 		&    0 	&   7 	&   	306 		&  8.0e-05 	& 2.3e-06 	& 4.0e-05 	&  4 			\\ \hline
A.12  	&  2 	& 2 		&    0 	&   1 	&    	83 		&  0 				& 2.9e-06 	& 0 				&  1 			\\ \hline
A.13  	&  3 	& 3 		&    0 	&   4 &  	8000 	&  0 				& 8.9e-05 	& 1.2e-05 		&  1 			\\ \hline
A.14  	& 10 	& 10 		& 0.01 	&   1 	&    	69 		& 	0 				& 1.2e-06 		& 0 				&  1 			\\ \hline
A.15  	&  3 	& 6 		&    0 	&   3 &  		6000 	&  0 				& 1.6e-05 		& 0 				&  1 			\\ \hline
A.16a 	&  5 	& 5 		&   10 	&   9 &  		1895 	& 1.1e-06 		& 1.0e-06 		& 3.1e-05 		&  4 			\\
A.16b 	&  5 	& 5 		&   10 	&  12	&  	1288 	& 1.5e-06 		& 2.4e-06 	& 2.8e-05 	&  1 			\\
A.16c 	&  5 	& 5 		&   10 	&   8 	&  	1130 		& 7.1e-06 		& 1.7e-06 		& 5.1e-05 		&  1 			\\
A.16d 	&  5 	& 5 		&   10	& 	 F 	& 					& 					& 					& 					&  			\\ \hline
A.17  	&  2 	& 3 		&    0 	&   7	&  	2902 	& 0 				& 1.3e-06 		& 3.5e-05 	& 16 			\\ \hline
A.18  	&  2 	& 12 		&    0 	&  11 & 		11950 	& 4.5e-06 	& 1e-06 		& 8.5e-05 	& 64 		\\
      		&    	&     		&    1 	&  11 & 		11950 	& 4.5e-06 	& 1e-06 		& 8.5e-05 	& 64 		\\
      		&    	&     		&   10 	&  11 & 		11945 	& 4.5e-06 	& 1e-06 		& 8.5e-05 	& 64 		\\ \hline
\end{tabular}
\end{table}
We summarize the results in Table~\ref{tab:AMPAL} and make several comments. With the exception of problem A.16d, Algorithm~\ref{alg:AMPAL} solves every problem given a proper initial point. Although the algorithm compares unfavorably to existing second-order methods for GNEPs in terms of solution accuracy, Algorithm~\ref{alg:AMPAL} is matrix-free and can scale to large problems. Further, the speed of Algorithm~\ref{alg:AMPAL} depends on how quickly the subproblems are solved. In this regard, the AMP algorithm benefits from avoiding computationally expensive operations, e.g., matrix inversion. However, as pointed by~\citet{Kanzow-2016-Augmented}, the subproblems have a semismooth structure in which the Levenberg-Marquardt algorithm can be superlinearly convergent in practice. As such, it is promising to investigate whether or not we can design first-order algorithms that benefit from such structure. It is also important to make good choices of the parameters that determine the multipliers and penalty parameters since they greatly affect the performance of the algorithm. For many application problems, we observe that fine-tuning the parameters can yield a speed improvement.

It is also worth mentioning that we have conducted numerical experiments with both Algorithm~\ref{alg:AMPQP} and Algorithm~\ref{alg:AMPAL} and Algorithm~\ref{alg:AMPAL} dominates Algorithm~\ref{alg:AMPQP}. This is consistent with the observation from~\citet{Kanzow-2016-Augmented} so we omit the results for Algorithm~\ref{alg:AMPQP}. It is important to acknowledge that existing second-order methods based on Levenberg-Marquardt or other linear system solvers are faster than our algorithms on small instances. This is because solving linear systems will not be an issue when the size of problem is small. However, these linear system solvers will be prohibitive as the problem size increases. In contrast, our methods do not need a sophisticated matrix factorization and can be thus scaled to large problems. 

\section{Conclusions}\label{sec:conclu}
We have presented an inquiry into equilibrium computation in a special class of nonlinear generalized Nash equilibrium problems (NGNEPs). Focusing on simple gradient-based schemes, we investigated algorithms based on a quadratic penalty method and an augmented Lagrangian method, respectively. Both of these algorithms make use of the accelerated mirror-prox algorithm as an inner loop. We established global convergence rate estimates for both algorithms for solving monotone and strongly monotone NGNEPs: the complexity bounds are $\tilde{O}(\epsilon^{-1})$ and $\tilde{O}(\epsilon^{-1/2})$ in monotone and strongly monotone cases, respectively, in terms of the number of gradient evaluations. We highlighted the practical relevance of these theoretical rates in experiments with several real-world datasets.

It is worth mentioning that there are many numerical variations of our work that might lead to quantitative improvements. Aside from fine-tuning of parameters, a more detailed analysis of penalized NEPs which occur in our method may shed light on the possibility of a more effective subproblem solver than accelerated mirror-prox. Future directions include the identification of other NGNEPs which can be solved by an algorithm with a global convergence rate estimate and the investigation of the sequential quadratic programming (SQP) method~\citep{Tolle-1995-Sequential} for solving GNEPs.

\section{Acknowledgment}
This work was supported in part by the Mathematical Data Science program of the Office of Naval Research under grant number N00014-18-1-2764 and by the Vannevar Bush Faculty Fellowship program under grant number N00014-21-1-2941.

\bibliographystyle{plainnat}
\bibliography{ref}

\appendix
\section{Proof of Theorem~\ref{Thm:NGNEP-KKT}}\label{app:NGNEP-KKT}
We provide the complete proof of Theorem~\ref{Thm:NGNEP-KKT}. We first present a simple lemma and prove the theorem by appealing to Farkas' lemma~\citep{Farkas-1902-Theorie}.
\begin{lemma}\label{Lemma:Farkas}
Let $A \in \br^{m \times n}$, $E \in \br^{e \times n}$ and $b \in \br^n$. Exactly one of the following systems has a solution: 
\begin{itemize}
\item $Ax \leq 0$, $Ex = 0$ and $b^\top x > 0$;
\item $A^\top y + E^\top z - b = 0$ and $y \geq 0$. 
\end{itemize}
\end{lemma}
\begin{proof}
Recall that Farkas' lemma states that exactly one of the following systems has a solution: 
\begin{itemize}
\item $\bar{A}x \geq 0$ and $\bar{b}^\top x < 0$;
\item $\bar{A}^\top y = b$ and $y \geq 0$. 
\end{itemize}
We thus rewrite two linear systems from Lemma~\ref{Lemma:Farkas} as follows, 
\begin{itemize}
\item $\begin{bmatrix} -A & E & -E \end{bmatrix}x \geq 0$ and $(-b)^\top x < 0$. 
\item $(-A)^\top y + E^\top\max\{0, -z\} + (-E)^\top\max\{0, z\} = -b$ and $y \geq 0$. 
\end{itemize}
Applying Farkas' lemma with $\bar{A} = \begin{bmatrix} -A & E & -E \end{bmatrix}$ and $\bar{b} = -b$ yields the desired result.
\end{proof}
\paragraph{Necessity.} By Proposition~\ref{Prop:NGNEP-equivalence}, a point $\bar{x} \in \hat{X}$ being a solution of a monotone NGNEP implies that $\bar{x} \in \hat{X}$ satisfies that $\bar{x}^\nu \in X_\nu(\bar{x}^{-\nu})$ and $(x - \bar{x})^\top v(\bar{x}) \geq 0$ for all $x \in X_\nu(\bar{x})$. By the definition of $v(\cdot)$ and $X_\nu(\cdot)$, we have
\begin{equation*}
(x^\nu -\bar{x}^\nu)^\top \nabla_\nu \theta_\nu(\bar{x}) \geq 0, 
\end{equation*}
for all $x^\nu \in \hat{X}_\nu$ satisfying that 
\begin{equation}\label{SM1-feasibility-condition-1}
\left\{\begin{array}{ll}
A_s^\nu x^\nu + \sum_{j \neq \nu, j \in \NCal_s} A_s^j \bar{x}^j - b_s \leq 0, & \forall s \in \ICal_\nu, \\
E_s^\nu x^\nu + \sum_{j \neq \nu, j \in \NCal_s} E_s^j \bar{x}^j - d_s = 0, & \forall s \in \ICal_\nu.
\end{array}\right.
\end{equation}
Equivalently, by letting $\bar{\xi}^\nu$ be an element of the normal cone of $\hat{X}_\nu$ at an $\bar{x}^\nu$, we have
\begin{equation*}
(x^\nu -\bar{x}^\nu)^\top(\nabla_\nu \theta_\nu(\bar{x}) + \bar{\xi}^\nu) \geq 0, 
\end{equation*}
for all $x^\nu \in \br^{n_\nu}$ satisfying the feasibility conditions in Eq.~\eqref{SM1-feasibility-condition-1}. By the change of variables $y^\nu = x^\nu -\bar{x}^\nu$, we have
\begin{equation*}
(y^\nu)^\top(\nabla_\nu \theta_\nu(\bar{x}) + \bar{\xi}^\nu) \geq 0, 
\end{equation*}
for all $y^\nu \in \br^{n_\nu}$ satisfying 
\begin{equation*}
A_s^\nu y^\nu \leq b_s - A_s\bar{x}^{\NCal_s}, \quad E_s^\nu y^\nu = 0, \quad \textnormal{for all } s \in \ICal_\nu.
\end{equation*}
Denote the index set of active inequality constraints by $I_s^\nu(\bar{x}) = \{i: (A_s\bar{x}^{\NCal_s} - b_s)_i = 0\}$ for all $s \in \ICal_\nu$, where $(\cdot)_i$ stands for the $i$-th element of a vector. Then, we have
\begin{equation*}
A_s^\nu y^\nu \leq b_s - A_s\bar{x}^{\NCal_s} \ \Longleftrightarrow \ 
\left\{\begin{array}{ll}
(A_s^\nu y^\nu)_i \leq 0, & \forall i \in I_s^\nu(\bar{x}), \\
(A_s^\nu y^\nu)_i \leq (b_s - A_s\bar{x}^{\NCal_s})_i, & \forall i \notin I_s^\nu(\bar{x}). 
\end{array}\right. 
\end{equation*}
Putting these pieces together yields that, if a point $\bar{x} \in \hat{X}$ is a solution of NGNEP, then the following statement is true:  
\begin{equation*}
(y^\nu)^\top(\nabla_\nu \theta_\nu(\bar{x}) + \bar{\xi}^\nu) \geq 0, 
\end{equation*}
for all $y^\nu \in \br^{n_\nu}$ satisfying the following conditions:
\begin{equation}\label{SM1-feasibility-condition-2}
\left\{\begin{array}{lll}
(A_s^\nu y^\nu)_i \leq 0, & \forall i \in I_s^\nu(\bar{x}), & \forall s \in \ICal_\nu, \\
(A_s^\nu y^\nu)_i \leq (b_s - A_s\bar{x}^{\NCal_s})_i, & \forall i \notin I_s^\nu(\bar{x}), & \forall s \in \ICal_\nu, \\
E_s^\nu y^\nu = 0, & & \forall s \in \ICal_\nu. 
\end{array}\right. 
\end{equation}
Suppose that $z^\nu \in \br^{n_\nu}$ satisfies that 
\begin{equation*}
\left\{\begin{array}{lll}
(A_s^\nu z^\nu)_i \leq 0, & \forall i \in I_s^\nu(\bar{x}), & \forall s \in \ICal_\nu, \\
E_s^\nu z^\nu = 0, & & \forall s \in \ICal_\nu. 
\end{array}\right. 
\end{equation*}
Since $(b_s - A_s\bar{x}^{\NCal_s})_i > 0$ for all $i \notin I_s^\nu(\bar{x})$ and all $s \in \ICal_\nu$, it follows there exists a sufficiently small scalar $\tau > 0$ such that $\tau(A_s^\nu z^\nu)_i \leq (b_s - A_s\bar{x}^{\NCal_s})_i$. This together with the definition of $z^\nu$ yields that $\bar{z}^\nu = \tau z^\nu$ satisfies Eq.~\eqref{SM1-feasibility-condition-2}. As such, we have $(\tau z^\nu)^\top(\nabla_\nu \theta_\nu(\bar{x}) + \bar{\xi}^\nu) \geq 0$. Since $\tau > 0$, we have $(z^\nu)^\top(\nabla_\nu \theta_\nu(\bar{x}) + \bar{\xi}^\nu) \geq 0$. In summary, we have
\begin{equation*}
\left\{\begin{array}{lll}
(A_s^\nu z^\nu)_i \leq 0, & \forall i \in I_s^\nu(\bar{x}), & \forall s \in \ICal_\nu \\
E_s^\nu z^\nu = 0, & & \forall s \in \ICal_\nu 
\end{array}\right. \ \Longrightarrow \ (z^\nu)^\top(\nabla_\nu \theta_\nu(\bar{x}) + \bar{\xi}^\nu) \geq 0. 
\end{equation*}
This implies that the following system does not have a solution:  
\begin{equation*}
\left\{\begin{array}{lll}
(z^\nu)^\top(-\nabla_\nu \theta_\nu(\bar{x}) - \bar{\xi}^\nu) > 0, & & \\  
(A_s^\nu z^\nu)_i \leq 0, & \forall i \in I_s^\nu(\bar{x}), & \forall s \in \ICal_\nu, \\
E_s^\nu z^\nu = 0, & & \forall s \in \ICal_\nu. 
\end{array}\right.
\end{equation*}
Thus, by Lemma~\ref{Lemma:Farkas}, there exist $\bar{\mu}^{\nu, s}$ and $\bar{\lambda}^{\nu, s}_i \geq 0$ such that for all $i \in I_s^\nu(\bar{\x})$ and all $s \in \ICal_\nu$ the following equality holds true:
\begin{equation*}
\nabla_\nu \theta_\nu(\bar{x}) + \bar{\xi}^\nu + \sum_{s \in \ICal_\nu} \left(\sum_{i \in I_s^\nu(\bar{x})} \bar{\lambda}^{\nu, s}_i a_i^{\nu, s} + (E_s^\nu)^\top\bar{\mu}^{\nu, s} \right) = 0, 
\end{equation*}
where $a_i^{\nu, s}$ is the $i^\textnormal{th}$ column of $(A_s^\nu)^\top$. Equivalently, by letting $\bar{\lambda}^{\nu, s}_i = 0$ for all $i \notin I_s^\nu(\bar{x})$ and all $s \in \ICal_\nu$, we have
\begin{equation*}
\nabla_\nu \theta_\nu(\bar{x}) + \bar{\xi}^\nu + \sum_{s \in \ICal_\nu} ((A_s^\nu)^\top\bar{\lambda}^{\nu, s} + (E_s^\nu)^\top\bar{\mu}^{\nu, s}) = 0, \quad (\bar{\lambda}^{\nu, s})^\top(A_s\bar{x}^{\NCal_s} - b_s) = 0, \quad \textnormal{for all } s \in \ICal_\nu. 
\end{equation*}
By the definition of $\bar{\xi}^\nu$, we have
\begin{equation*}
(x^\nu -\bar{x}^\nu)^\top\left(\nabla_\nu \theta_\nu(\bar{x}) + \sum_{s \in \ICal_\nu} ((A_s^\nu)^\top\bar{\lambda}^{\nu, s} + (E_s^\nu)^\top\bar{\mu}^{\nu, s})\right) \geq 0, \textnormal{ for all } x^\nu \in \hat{X}_\nu. 
\end{equation*}
Since $\bar{x}^\nu \in X_\nu(\bar{x}^{-\nu})$, we have
\begin{equation*}
A_s\bar{x}^{\NCal_s} - b_s \leq 0, \quad E_s\bar{x}^{\NCal_s} - d_s = 0, \quad \textnormal{for all } s \in \ICal_\nu. 
\end{equation*}
Putting these pieces together yields that $\bar{x} \in \hat{X}$ is a KKT point. 
\paragraph{Sufficiency.} Suppose that a point $\bar{x} = (x^1, x^2, \ldots, x^N) \in \hat{X}$ is a KKT point so that there exist Lagrangian multipliers $\bar{\lambda}^{\nu, s} \geq 0$ and $\bar{\mu}^{\nu, s}$ satisfying that 
\begin{equation}\label{SM1-feasibility-condition-3}
\begin{array}{rcll}
(x^\nu -\bar{x}^\nu)^\top(\nabla_\nu \theta_\nu(\bar{x}) + \sum_{s \in \ICal_\nu} ((A_s^\nu)^\top\bar{\lambda}^{\nu, s} + (E_s^\nu)^\top\bar{\mu}^{\nu, s})) & \geq & 0, & \forall x^\nu \in \hat{X}_\nu, \\
A_s\bar{x}^{\NCal_s} - b_s & \leq & 0, & \forall s \in \ICal_\nu, \\
E_s\bar{x}^{\NCal_s} - d_s & = & 0, & \forall s \in \ICal_\nu, \\
(\bar{\lambda}^{\nu, s})^\top(A_s\bar{x}^{\NCal_s} - b_s) & = & 0, & \forall s \in \ICal_\nu.    
\end{array}
\end{equation}
Define the function $L_\nu(x^\nu)$ (which is convex due to the monotonicity) by 
\begin{eqnarray*}
\lefteqn{L_\nu(x^\nu) = \theta_\nu(x^\nu; \bar{x}^{-\nu})} \\
& & + \sum_{s \in \ICal_\nu} \left(\left(A_s^\nu x^\nu + \sum_{j \neq \nu, j \in \NCal_s} A_s^j \bar{x}^j  - b_s\right)^\top\bar{\lambda}^{\nu, s} + \left(E_s^\nu x^\nu + \sum_{j \neq \nu, j \in \NCal_s} E_s^j \bar{x}^j  - d_s\right)^\top\bar{\mu}^{\nu, s}\right). 
\end{eqnarray*}
By the definition of $L_\nu$ and using the third and fourth equations in Eq.~\eqref{SM1-feasibility-condition-3}, we have
\begin{equation*}
L_\nu(\bar{x}^\nu) = \theta_\nu(\bar{x}) + \sum_{s \in \ICal_\nu} ((\bar{\lambda}^{\nu, s})^\top(A_s\bar{x}^{\NCal_s} - b_s) + (\bar{\mu}^{\nu, s})^\top(E_s\bar{x}^{\NCal_s} - d_s)) = \theta_\nu(\bar{x}). 
\end{equation*}
The first inequality in Eq.~\eqref{SM1-feasibility-condition-3} implies that $(x^\nu -\bar{x}^\nu)^\top \nabla L_\nu(\bar{x}^\nu) \geq 0$ for all $x^\nu \in \hat{X}_\nu$ and thus we have $L_\nu(\bar{x}^\nu) = \min_{x^\nu \in \hat{X}_\nu} L_\nu(x^\nu)$. Putting these pieces together yields that 
\begin{equation*}
\theta_\nu(\bar{x}) \leq L_\nu(x^\nu), \quad \textnormal{for all } x^\nu \in \hat{X}_\nu. 
\end{equation*}
Suppose that $x^\nu \in \hat{X}_\nu$ is a feasible point in the sense that $x^\nu \in X_\nu(\bar{x}^{-\nu}) \subseteq \hat{X}_\nu$, we have
\begin{equation*}
\begin{array}{rcll}
A_s^\nu x^\nu + \sum_{j \neq \nu, j \in \NCal_s} A_s^j \bar{x}^j  - b_s & \leq & 0, & \forall s \in \ICal_\nu, \\ 
E_s^\nu x^\nu + \sum_{j \neq \nu, j \in \NCal_s} E_s^j \bar{x}^j  - d_s & = & 0, & \forall s \in \ICal_\nu. 
\end{array}
\end{equation*}
By the definition of $L_\nu$ and some simple calculations, we have $L_\nu(x^\nu) \leq \theta_\nu(x^\nu, \bar{x}^{-\nu})$. Thus, we have
\begin{equation*}
\theta_\nu(\bar{x}^\nu, \bar{x}^{-\nu}) \leq \theta_\nu(x^\nu, \bar{x}^{-\nu}), \quad \textnormal{for all } x^\nu \in X_\nu(\bar{x}^{-\nu}). 
\end{equation*}
Repeating the above argument for each player $\nu \in \NCal$, we conclude from Definition~\ref{def:NGNEP-solution} that $\bar{x} \in \hat{X}$ is a solution. 

\section{Variational Equilibrium}
There is an important subclass of solutions, referred to as \emph{variational solutions} (or normalized equilibrium), which is characterized by the condition that the same Lagrange multipliers are associated with the constraints in each player's problem. This condition is motivated by economic models where the Lagrange multipliers in each player's problem have an interpretation as shadow prices of the resources and are thus equal.  Computational algorithms have been developed that specifically target the variational solutions of GNEPs, and these algorithms have been analyzed in terms of asymptotic global convergence.  But iteration complexity analyses have not yet been provided for these algorithms. 
\begin{definition}\label{def:VNE}
We say a point $\bar{\x}$ is a variational solution of an NGNEP if 
\begin{equation*}
\bar{x} \in \hat{X} \cap \left\{x \in \br^n: A_s x^{\NCal_s} \leq b_s, \ E_s x^{\NCal_s} = d_s, \textnormal{ for all } s \in [S]\right\}, 
\end{equation*}
and there exist some Lagrangian multipliers $\bar{\lambda}^s \geq 0$ and $\bar{\mu}^s$ such that 
\begin{equation*}
\textnormal{ for all } \nu \in \NCal \textnormal{ we have } \left\{\begin{array}{rl}
(\bar{x}^\nu - x^\nu)^\top(v_\nu(\bar{x}) + \sum_{s \in \ICal_\nu} ((A_s^\nu)^\top\bar{\lambda}^s + (E_s^\nu)^\top\bar{\mu}^s)) \leq 0, & \forall x^\nu \in \hat{X}_\nu, \\
(\bar{\lambda}^s)^\top(A_s\bar{x}^{\NCal_s} - b_s) = 0, & \forall s \in \ICal_\nu. 
\end{array}\right.
\end{equation*}
\end{definition}
Based on the KKT conditions, the set of variational equilibria of monotone NGNEPs coincides with the solution set of the following VI: 
\begin{equation}\label{def:NVI}
\textnormal{Find } \bar{x} \in X: \quad (x - \bar{x})^\top v(\bar{x}) \geq 0, \textnormal{ for all } x \in X,  
\end{equation}
where $X$ is defined by 
\begin{equation*}
X = \hat{X} \cap \left\{x \in \br^n: A_s x^{\NCal_s} \leq b_s, \ E_s x^{\NCal_s} = d_s, \textnormal{ for all } s \in [S]\right\}.  
\end{equation*}
By reading off classical results from the VI literature, we can derive an existence result for variational solutions of NGNEPs (cf. Definition~\ref{def:NGNEPs}) due to the compactness of $\hat{X}$. Formally, we have
\begin{proposition}\label{Prop:VNE}
There exists at least one variational solution of any NGNEP. 
\end{proposition}
\begin{proof}
Since at least one solution of the NGNEP exists, $X$ will be an nonempty, convex and compact set. Therefore, the VI in Eq.~\eqref{def:NVI} must have a solution and hence there exists at least one variational equilibrium of the NGNEP in the sense of Definition~\eqref{def:VNE}. 
\end{proof}

\section{The AMP Algorithm}\label{app:AMP-proof}
We provide the proof of Theorem~\ref{Thm:AMP-SMT} that delineates an iteration complexity bound of the AMP algorithm for solving \textit{composite} smooth and strongly monotone VIs. More specifically, the scheme aims at solving the following class of VIs: 
\begin{equation}\label{app:CMVI}
\textnormal{Find } z^\star \in Z: \quad (z - z^\star)^\top(F(z^\star) + \nabla G(z^\star)) \geq 0, \quad \textnormal{for all } z \in Z,  
\end{equation} 
where $F: Z \mapsto \br^n$ is $\ell_F$-Lipschitz and $\alpha$-strongly monotone, $G: Z \mapsto \br$ is $\ell_G$-smooth and convex, and $Z$ is an nonempty, closed and convex set with diameter $D_Z > 0$. With the initialization $z_1 = z_1^{\textnormal{ag}} = w_1 \in Z$, a typical iteration of the algorithm is given by
\begin{equation}\label{alg:AMP}
\begin{array}{lcl}
z_k^{\textnormal{md}} & = & (1-\alpha_k)z_k^{\textnormal{ag}} + \alpha_k w_k, \\
z_{k+1} & = & \argmin_{z \in Z} \gamma_k(z- w_k)^\top(F(w_k) + \nabla G(z_k^{\textnormal{md}})) + \tfrac{1}{2}\|z - w_k\|^2, \\
w_{k+1} & = & \argmin_{z \in Z} \gamma_k(z- w_k)^\top(F(z_{k+1}) + \nabla G(z_k^{\textnormal{md}})) + \tfrac{1}{2}\|z - w_k\|^2, \\
z_{k+1}^{\textnormal{ag}} & = & (1-\alpha_k)z_k^{\textnormal{ag}} + \alpha_k z_{k+1}, 
\end{array}
\end{equation} 
where $\alpha_k \in (0, 1)$ and $\gamma_k > 0$ are parameters to be determined. 

\subsection{Technical lemmas}
We present some technical lemmas which are important for the subsequent analysis. To state these lemmas it is useful to define the following \emph{gap function}:  
\begin{equation*}
Q(\tilde{z}, z) = G(\tilde{z}) - G(z) + (\tilde{z} - z)^\top F(z). 
\end{equation*}
Our first lemma gives a descent inequality for the iterates generated by the AMP algorithm (cf. Eq.~\eqref{alg:AMP}). This is a consequence of~\citet[Lemma~6.3]{Juditsky-2011-Solving}; we provide a proof for completeness. 
\begin{lemma}\label{Lemma:AMP-descent}
Suppose that the iterates $\{(z_k^{\textnormal{md}}, z_k, w_k, z_k^{\textnormal{ag}})\}_{k \geq 1}$ are generated by the AMP algorithm (cf. Eq.~\eqref{alg:AMP}). Then, the following statement holds true for all $z \in Z$: 
\begin{equation*}
\gamma_k(z_{k+1} - z)^\top(F(z_{k+1}) + \nabla G(z_k^{\textnormal{md}})) \leq  \tfrac{1}{2}(\|z - w_k\|^2 - \|z - w_{k+1}\|^2) - (\tfrac{1}{2} - \tfrac{\gamma_k^2\ell_F^2}{2})\|z_{k+1} - w_k\|^2. 
\end{equation*}
\end{lemma}
\begin{proof}
From the update rule for $z_{k+1}$ and $w_{k+1}$, we have
\begin{align*}
\gamma_k(z_{k+1} - z)^\top(F(w_k) + \nabla G(z_k^{\textnormal{md}})) & \leq \tfrac{1}{2}(\|z - w_k\|^2 -\|z_{k+1} - w_k\|^2 - \|z - z_{k+1}\|^2), \\
\gamma_k(w_{k+1} - z)^\top(F(z_{k+1}) + \nabla G(z_k^{\textnormal{md}})) & \leq \tfrac{1}{2}(\|z - w_k\|^2 - \|w_{k+1} - w_k\|^2 - \|z - w_{k+1}\|^2). 
\end{align*}
Setting $z = w_{k+1}$ in the first inequality and adding the resulting inequality to the second inequality yields that 
\begin{eqnarray*}
\lefteqn{\gamma_k(z_{k+1} - w_{k+1})^\top(F(w_k) + \nabla G(z_k^{\textnormal{md}})) + \gamma_k(w_{k+1} - z)^\top(F(z_{k+1}) + \nabla G(z_k^{\textnormal{md}}))} \\
& \leq & \tfrac{1}{2}(\|z - w_k\|^2 - \|z - w_{k+1}\|^2 - \|z_{k+1} - w_k\|^2 - \|w_{k+1} - z_{k+1}\|^2). 
\end{eqnarray*}
Equivalently, we have
\begin{eqnarray*}
\lefteqn{\gamma_k(z_{k+1} - z)^\top(F(z_{k+1}) + \nabla G(z_k^{\textnormal{md}})) \leq \gamma_k(z_{k+1} - w_{k+1})^\top(F(z_{k+1}) - F(w_k))} \\
& & + \tfrac{1}{2}(\|z - w_k\|^2 - \|z - w_{k+1}\|^2 - \|z_{k+1} - w_k\|^2 - \|w_{k+1} - z_{k+1}\|^2). 
\end{eqnarray*}
By Young's inequality and the fact that $F$ is $\ell_F$-Lipschitz, we have
\begin{equation*}
(z_{k+1} - w_{k+1})^\top(F(z_{k+1}) - F(w_k)) \leq \tfrac{1}{2\gamma_k}\|z_{k+1} - w_{k+1}\|^2 + \tfrac{\gamma_k\ell_F^2}{2}\|z_{k+1} - w_k\|^2. 
\end{equation*}
Putting these pieces together yields the desired inequality. 
\end{proof}
\begin{lemma}\label{Lemma:AMP-gap}
Suppose that the iterates $\{(z_k^{\textnormal{md}}, z_k, w_k, z_k^{\textnormal{ag}})\}_{k \geq 1}$ are generated by the AMP algorithm (cf. Eq.~\eqref{alg:AMP}). Then, the following statement holds true for all $z \in Z$:
\begin{eqnarray*}
\lefteqn{Q(z_{k+1}^{\textnormal{ag}}, z) - (1-\alpha_k)Q(z_k^{\textnormal{ag}}, z)} \\
& \leq & \alpha_k(z_{k+1} - z)^\top(F(z_{k+1}) + \nabla G(z_k^{\textnormal{md}})) + \tfrac{\alpha_k^2\ell_G}{2}\|z_{k+1} - w_k\|^2 - \alpha_k\alpha\|z_{k+1} - z\|^2.
\end{eqnarray*}
\end{lemma}
\begin{proof}
From the update rule for $z_k^{\textnormal{md}}$ and $z_{k+1}^{\textnormal{ag}}$, we have $z_{k+1}^{\textnormal{ag}} - z_k^{\textnormal{md}} = \alpha_k(z_{k+1} - w_k)$. Since $G$ is $\ell_G$-smooth, we have
\begin{eqnarray*}
\lefteqn{G(z_{k+1}^{\textnormal{ag}}) \leq G(z_k^{\textnormal{md}}) + (z_{k+1}^{\textnormal{ag}} - z_k^{\textnormal{md}})^\top \nabla G(z_k^{\textnormal{md}}) + \tfrac{\ell_G}{2}\|z_{k+1}^{\textnormal{ag}} - z_k^{\textnormal{md}}\|^2} \\
& \leq & G(z_k^{\textnormal{md}}) + (z_{k+1}^{\textnormal{ag}} - z_k^{\textnormal{md}})^\top \nabla G(z_k^{\textnormal{md}}) + \tfrac{\alpha_k^2\ell_G}{2}\|z_{k+1} - w_k\|^2. 
\end{eqnarray*}
Since $z_{k+1}^{\textnormal{ag}} = (1-\alpha_k)z_k^{\textnormal{ag}} + \alpha_k z_{k+1}$, we have
\begin{eqnarray*}
\lefteqn{G(z_{k+1}^{\textnormal{ag}}) \leq (1-\alpha_k)\left(G(z_k^{\textnormal{md}}) + (z_k^{\textnormal{ag}} - z_k^{\textnormal{md}})^\top \nabla G(z_k^{\textnormal{md}})\right)} \\
& +\, \alpha_k\left(G(z_k^{\textnormal{md}}) + (z - z_k^{\textnormal{md}})^\top \nabla G(z_k^{\textnormal{md}})\right) + \alpha_k(z_{k+1} - z)^\top \nabla G(z_k^{\textnormal{md}}) + \tfrac{\alpha_k^2\ell_G}{2}\|z_{k+1} - w_k\|^2. 
\end{eqnarray*}
Since $G$ is convex, we have
\begin{equation*}
G(z_{k+1}^{\textnormal{ag}}) \leq (1-\alpha_k)G(z_k^{\textnormal{ag}}) + \alpha_k G(z) + \alpha_k(z_{k+1} - z)^\top \nabla G(z_k^{\textnormal{md}}) + \tfrac{\alpha_k^2\ell_G}{2}\|z_{k+1} - w_k\|^2. 
\end{equation*}
Combining the definition of $Q$ with the above inequality and $z_{k+1}^{\textnormal{ag}} = (1-\alpha_k)z_k^{\textnormal{ag}} + \alpha_k z_{k+1}$, we have
\begin{eqnarray*}
\lefteqn{Q(z_{k+1}^{\textnormal{ag}}, z) - (1-\alpha_k)Q(z_k^{\textnormal{ag}}, z)} \\
& \leq & G(z_{k+1}^{\textnormal{ag}}) - (1-\alpha_k)G(z_k^{\textnormal{ag}}) - \alpha_k G(z) + (z_{k+1}^{\textnormal{ag}} - z)^\top F(z) - (1-\alpha_k)(z_k^{\textnormal{ag}} - z)^\top F(z) \\ 
& \leq & \alpha_k(z_{k+1} - z)^\top \nabla G(z_k^{\textnormal{md}}) + \tfrac{\alpha_k^2\ell_G}{2}\|z_{k+1} - w_k\|^2 + (z_{k+1}^{\textnormal{ag}} - z - (1-\alpha_k)(z_k^{\textnormal{ag}} - z))^\top F(z) \\
& = & \alpha_k(z_{k+1} - z)^\top \nabla G(z_k^{\textnormal{md}}) + \tfrac{\alpha_k^2\ell_G}{2}\|z_{k+1} - w_k\|^2 + \alpha_k(z_{k+1} - z)^\top F(z).    
\end{eqnarray*}
Since $F$ is $\alpha$-strongly monotone, we have 
\begin{equation*}
(z_{k+1} - z)^\top F(z) \leq (z_{k+1} - z)^\top F(z_{k+1}) - \alpha\|z_{k+1} - z\|^2. 
\end{equation*}
Putting these pieces together yields the desired inequality. 
\end{proof}
\subsection{Proof of Theorem~\ref{Thm:AMP-SMT}}
Combining Lemma~\ref{Lemma:AMP-descent} and Lemma~\ref{Lemma:AMP-gap}, we have
\begin{eqnarray*}
\lefteqn{Q(z_{k+1}^{\textnormal{ag}}, z) - (1-\alpha_k)Q(z_k^{\textnormal{ag}}, z)} \\
& \leq & \tfrac{\alpha_k}{\gamma_k}\left(\tfrac{1}{2}(\|z - w_k\|^2 - \|z - w_{k+1}\|^2) - (\tfrac{1}{2} - \tfrac{\gamma_k^2\ell_F^2}{2})\|z_{k+1} - w_k\|^2\right) + \tfrac{\alpha_k^2\ell_G}{2}\|z_{k+1} - w_k\|^2 \\ 
& & - \alpha_k\alpha\|z_{k+1} - z\|^2, \\
& = & \tfrac{\alpha_k}{2\gamma_k}\|z - w_k\|^2 - \tfrac{\alpha_k}{2\gamma_k}\|z - w_{k+1}\|^2 - (\tfrac{\alpha_k}{2\gamma_k} - \tfrac{\alpha_k\gamma_k\ell_F^2}{2} - \tfrac{\alpha_k^2\ell_G}{2})\|z_{k+1} - w_k\|^2 - \alpha_k\alpha\|z_{k+1} - z\|^2. 
\end{eqnarray*}
Since $\alpha_k \equiv \alpha_0 = \tfrac{1}{4}\min\{\tfrac{\alpha}{\ell_F}, \sqrt{\tfrac{\alpha}{\ell_G}}\}$ and $\gamma_k = \tfrac{\alpha_k}{\alpha}$ for all $k \geq 1$, we remove the subscript in the above inequality and define the energy term $E_k = Q(z_k^{\textnormal{ag}}, z) + \tfrac{\alpha}{2}\|z - w_k\|^2$. We have
\begin{equation*}
E_{k+1} - (1-\alpha_0)E_k \leq \tfrac{\alpha_0\alpha}{2}\|z - w_k\|^2 - (\tfrac{\alpha}{2} - \tfrac{\alpha_0^2\ell_F^2}{2\alpha} - \tfrac{\alpha_0^2\ell_G}{2})\|z_{k+1} - w_k\|^2 - \alpha_0\alpha\|z_{k+1} - z\|^2. 
\end{equation*}
By Young's inequality, we have
\begin{equation*}
E_{k+1} - (1-\alpha_0)E_k \leq - (\tfrac{\alpha}{2} - \alpha_0\alpha - \tfrac{\alpha_0^2\ell_F^2}{2\alpha} - \tfrac{\alpha_0^2\ell_G}{2})\|z_{k+1} - w_k\|^2. 
\end{equation*}
By the definition of $\alpha_0 = \tfrac{1}{4}\min\{\frac{\alpha}{\ell_F}, \sqrt{\frac{\alpha}{\ell_G}}\}$, we have
\begin{equation*}
\alpha\alpha_0 \leq \tfrac{\alpha}{4}, \quad \tfrac{\alpha_0^2\ell_F^2}{2\alpha} \leq \tfrac{\alpha}{8}, \quad \tfrac{\alpha_0^2\ell_G}{2} \leq \tfrac{\alpha}{8}. 
\end{equation*}
Putting these pieces together yields that $E_{k+1} \leq (1-\alpha_0)E_k$ for all $k \geq 1$. Then, by using the definition of $E_k$, we have
\begin{equation*}
\max_{z \in Z} Q(z_k^{\textnormal{ag}}, z) \leq \left(1-\tfrac{1}{4}\min\left\{\tfrac{\alpha}{\ell_F}, \sqrt{\tfrac{\alpha}{\ell_G}}\right\}\right)^{k-1} \max_{z \in Z} \ \{Q(z_1^{\textnormal{ag}}, z) + \tfrac{\alpha}{2}\|z - w_1\|^2\}. 
\end{equation*}
This together with the fact that the diameter of $Z$ is $D_Z$ implies the desired inequality. 

\section{Iteration Complexity Analysis of the AMPQP Algorithm}\label{app:AMPQP-proof}
We provide the proof of Theorem~\ref{Thm:AMPQP} on the iteration complexity bound of the AMPQP algorithm (cf. Algorithm~\ref{alg:AMPQP}) for solving monotone and strongly monotone NGNEPs. 

\subsection{Technical lemmas}
We present some technical lemmas which are important to the subsequent analysis.  Consider a quadratic penalization of an NGNEP where each player's minimization problem is given by  
\begin{equation*}
\min_{x^\nu \in \hat{X}^\nu} \theta_\nu(x^\nu, x^{-\nu}) +  \left(\sum_{s \in \ICal_\nu} \tfrac{\beta^s}{2}\|\max\{0, A_s x^{\NCal_s}-b_s\}\|^2 + \tfrac{\rho^s}{2}\|E_s x^{\NCal_s}-d_s\|^2\right), 
\end{equation*}
where $(\beta^s, \rho^s) \in \br_+ \times \br_+$ stand for penalty parameters associated with inequality and equality constraints. We define the functions $g_1: \br^n \mapsto \br$ and $h_1: \br^n \mapsto \br$ by
\begin{equation*}
g_1(x) = \sum_{s=1}^S \tfrac{\beta^s}{2}\|\max\{0, A_s x^{\NCal_s}-b_s\}\|^2, \quad h_1(x) = \sum_{s=1}^S \tfrac{\rho^s}{2}\|E_s x^{\NCal_s}-d_s\|^2. 
\end{equation*}
Concatenating the first-order optimality conditions of Eq.~\eqref{prob:NGNEP-QP} for each player, we aim to solve the following VI ($\hat{X}$ is convex and compact with a diameter $D>0$): 
\begin{equation}\label{app:PVI-QP}
\textnormal{Find } x \in \hat{X}: \quad (x' - x)^\top(v(x) + \nabla g_1(x) + \nabla h_1(x)) \geq 0, \ \textnormal{ for all } x' \in \hat{X}.     
\end{equation}
We first state the following well-known result which guarantees that the distance function has a Lipschitz-continuous gradient~\citep[see][Proposition~5]{Lan-2011-Primal}.
\begin{proposition}\label{Prop:smooth}
Given a closed convex set $\CCal \subseteq \br^n$, we let $d_\CCal: \br^n \mapsto \br$ be the distance function to $\CCal$ with respect to $\|\cdot\|$ on $\br^n$. Then, the function $\psi(\cdot) = (d_\CCal(\cdot))^2$ is convex and its gradient is given by $\nabla\psi(x) = 2(x-\PCal_\CCal(x))$ for all $x \in \br^n$. In addition, the gradient is Lipschitz continuous; that is, $\|\nabla\psi(\tilde{x}) - \nabla\psi(x)\| \leq 2\|\tilde{x} - x\|$ for all $\tilde{x}, x \in \br^n$. 
\end{proposition}
As an immediate consequence of Proposition~\ref{Prop:smooth}, we obtain the following lemma. 
\begin{lemma}\label{Lemma:smooth-QP}
The gradients of the functions $g_1$ and $h_1$ are $\ell_\beta$-Lipschitz continuous and $\ell_\rho$-Lipschitz continuous respectively, where $\ell_\beta \mydefn \sum_{s=1}^S \beta^s\|A_s\|^2$ and $\ell_\rho \mydefn \sum_{s=1}^S \rho^s\|E_s\|^2$. 
\end{lemma}
\begin{proof}
By the definition, the differentiability of $g_1$ and $h_1$ comes from Proposition~\ref{Prop:smooth}. By the chain rule and letting $d_\CCal$ be the distance function to $\CCal$ with respect to $\|\cdot\|$, we have
\begin{equation*}
\nabla_\nu g_1(x) = \sum_{s = 1}^S \beta^s (A_s^\nu)^\top(\max\{0, A_s x^{\NCal_s} - b_s\}),  \quad \nabla_\nu h_1(x) = \sum_{s = 1}^S \rho^s (E_s^\nu)^\top(E_s x^{\NCal_s}-d_s). 
\end{equation*}
Together with Proposition~\ref{Prop:smooth}, we have
\begin{equation*}
\begin{array}{rclcl}
\|\nabla g_1(\tilde{x})- \nabla g_1(x)\| & \leq & \sum_{s=1}^S \beta^s\|A_s\|\|A_s\tilde{x}^{\NCal_s}-A_s x^{\NCal_s}\| & \leq & \ell_\beta\|\tilde{x}-x\|, \\ 
\|\nabla h_1(\tilde{x})- \nabla h_1(x)\| & \leq & \sum_{s=1}^S \rho^s\|E_s\|\|E_s\tilde{x}^{\NCal_s}-E_s x^{\NCal_s}\| & \leq & \ell_\rho\|\tilde{x}-x\|. 
\end{array}
\end{equation*}
This completes the proof. 
\end{proof}
Combining the above lemma with Theorem~\ref{Thm:AMP-MT} and~\ref{Thm:AMP-SMT}, we obtain the following lemma. 
\begin{lemma}\label{Lemma:subroutine-QP}
Suppose that $x_{k+1} = \textsc{amp}((\beta_{k+1}^s, \rho_{k+1}^s)_{s \in [S]}, \delta_{k+1}, x_k)$ in the sense that $x_{k+1}$ is a $\delta_{k+1}$-approximate solution to the structured VI in Eq.~\eqref{app:PVI-QP} with $(\beta^s, \rho^s) = (\beta_{k+1}^s, \rho_{k+1}^s)$ such that 
\begin{equation*}
\max_{x \in \hat{X}} \left\{g_1(x_{k+1}) + h_1(x_{k+1}) - g_1(x) - h_1(x) + (x_{k+1} - x)^\top v(x)\right\} \leq \delta_{k+1}. 
\end{equation*}
The required number of gradient evaluations at the $k^\textnormal{th}$ iteration (for $k \geq 1$) is bounded by 
\begin{equation*}
N_k = \left\{\begin{array}{ll}
O\left(\sqrt{\frac{\sum_{s=1}^S (\beta_k^s\|A_s\|^2 + \rho_k^s\|E_s\|^2) D^2}{\delta_k}} + \tfrac{\sqrt{N}\ell_\theta D^2}{\delta_k}\right), & \textnormal{if } \alpha = 0, \\
O\left(\left(\sqrt{\frac{\sum_{s=1}^S (\beta_k^s\|A_s\|^2 + \rho_k^s\|E_s\|^2)}{\alpha}} + \tfrac{\sqrt{N}\ell_\theta}{\alpha}\right)\log\left(\tfrac{(\sqrt{N}\ell_\theta+\ell_G)D^2}{\delta_k}\right)\right), & \textnormal{if } \alpha > 0. 
\end{array}\right. 
\end{equation*}
\end{lemma}
\begin{proof}
By Lemma~\ref{Lemma:smooth-QP}, the VI in Eq.~\eqref{app:PVI-QP} is in the form of Eq.~\eqref{app:CMVI} with $\ell_F = \sqrt{N}\ell_\theta$ and $\ell_G = \sum_{s=1}^S (\beta_k^s\|A_s\|^2 + \rho_k^s\|E_s\|^2)$. Applying Theorem~\ref{Thm:AMP-MT} and~\ref{Thm:AMP-SMT}, we obtain the desired upper bound for the required number of gradient evaluations at the $k^\textnormal{th}$ iteration. 
\end{proof}
We continue with proving that $x_T$ is an $\hat{\epsilon}$-solution of NGNEPs for some $\hat{\epsilon} > 0$.   
\begin{lemma}\label{Lemma:main-QP}
Suppose that $x_T$ is a $\delta_T$-approximate solution to the structured VI in Eq.~\eqref{app:PVI-QP} with $(\beta^s, \rho^s) = (\beta_T^s, \rho_T^s)$ in the sense that 
\begin{equation*}
\max_{x \in \hat{X}} \left\{g_1(x_T) + h_1(x_T) - g_1(x) - h_1(x) + (x_T - x)^\top v(x)\right\} \leq \delta_T.  
\end{equation*}
Then, $x_T \in \hat{X}$ satisfies  
\begin{equation}\label{opt:eps-QP-feas}
\|\max\{0, A_s x_T^{\NCal_s} - b_s\}\| \leq \hat{\epsilon}_1, \quad \|E_s x_T^{\NCal_s} - d_s\| \leq \hat{\epsilon}_1, \quad \textnormal{for all } s \in [S], 
\end{equation}
where $\hat{\epsilon}_1 > 0$ is defined by 
\begin{equation*}
\hat{\epsilon}_1 = 2\sqrt{\tfrac{\delta_T + 10S\left(\max_{1 \leq s \leq S} \left\{\tfrac{\|\bar{\lambda}^s\|^2}{\beta_T^s} + \tfrac{\|\bar{\mu}^s\|^2}{\rho_T^s}\right\}\right)}{\min_{1 \leq s \leq S} \{\beta_T^s, \rho_T^s\}}}, 
\end{equation*}
and for all $\nu \in \NCal$, we have 
\begin{equation*}
(x_T^\nu - x^\nu)^\top v_\nu(x^\nu, x_T^{-\nu}) \leq \hat{\epsilon}_2, 
\end{equation*}
for all $x^\nu \in \hat{X}_\nu$ satisfying that 
\begin{equation}\label{opt:eps-QP-feas-VI}
\|\max\{0, A_s^\nu x^\nu + \sum_{j \in \NCal_s, j \neq \nu} A_s^j x_T^j - b_s\}\| \leq \epsilon, \quad \|E_s^\nu x^\nu + \sum_{j \in \NCal_s, j \neq \nu} E_s^j x_T^j - d_s\| \leq \epsilon, \quad \textnormal{for all } s \in \ICal_\nu, 
\end{equation}
where $\hat{\epsilon}_2 > 0$ is defined by 
\begin{equation*}
\hat{\epsilon}_2 = \delta_T + \tfrac{\epsilon^2}{2}\left(\sum_{s = 1}^S \beta_T^s + \rho_T^s\right). 
\end{equation*} 
\end{lemma}
\begin{proof}
By Proposition~\ref{Prop:VNE},  there exists a triple $(\bar{x} \in \hat{X}, \bar{\lambda} \geq 0, \bar{\mu}) $ such that  
\begin{equation}\label{condition:main-QP-first}
A_s \bar{x}^{\NCal_s} \leq b_s, \ E_s \bar{x}^{\NCal_s} = d_s, \ (\bar{\lambda}^s)^\top(A_s\bar{x}^{\NCal_s} - b_s) = 0, \ \textnormal{for all } s \in [S], 
\end{equation}
and the following VI holds true for all $\nu \in \NCal$, 
\begin{equation}\label{condition:main-QP-second}
(\bar{x}^\nu - x^\nu)^\top\left(v_\nu(\bar{x}) + \sum_{s \in \ICal_\nu} ((A_s^\nu)^\top\bar{\lambda}^s + (E_s^\nu)^\top\bar{\mu}^s)\right) \leq 0, \quad \textnormal{for all } x^\nu \in \hat{X}_\nu. 
\end{equation}
Note that Eq.~\eqref{condition:main-QP-first} implies that $g_1(\bar{x}) = 0$ and $h_1(\bar{x}) = 0$ and we have 
\begin{equation}\label{inequality:main-QP-third}
g_1(x_T) + h_1(x_T) - g_1(x) - h_1(x) + (x_T - x)^\top v(x) \leq \delta_T, \quad \textnormal{for all } x \in \hat{X}.   
\end{equation}
Plugging $x = \bar{x}$ into Eq.~\eqref{inequality:main-QP-third}, we have 
\begin{equation}\label{inequality:main-QP-fourth}
g_1(x_T) + h_1(x_T) + (x_T - \bar{x})^\top v(\bar{x}) \leq \delta_T. 
\end{equation}
Plugging $x^\nu = x_T^\nu$ into Eq.~\eqref{condition:main-QP-second} yields that 
\begin{equation*}
(\bar{x}^\nu - x_T^\nu)^\top \left(v_\nu(\bar{x}) + \sum_{s \in \ICal_\nu} (A_s^\nu)^\top\bar{\lambda}^s + (E_s^\nu)^\top\bar{\mu}^s\right) \leq 0, \quad \textnormal{ for all } \nu \in \NCal. 
\end{equation*}
Summing over $\nu \in \NCal$ and rearranging yields that 
\begin{equation*}
(\bar{x} - x_T)^\top v(\bar{x}) + \sum_{s=1}^S ((A_s\bar{x}^{\NCal_s} - A_s x_T^{\NCal_s})^\top \bar{\lambda}^s + (E_s\bar{x}^{\NCal_s} - E_s x_T^{\NCal_s})^\top \bar{\mu}^s) \leq 0.  
\end{equation*}
Combining this inequality with Eq.~\eqref{condition:main-QP-first} and the fact that $\bar{\lambda}^s \geq 0$, we have
\begin{eqnarray}\label{inequality:main-QP-key}
\lefteqn{(\bar{x} - x_T)^\top v(\bar{x}) \leq \sum_{s=1}^S (\|\max\{0, A_s x_T^{\NCal_s} - b_s\}\| \|\bar{\lambda}^s\| + \|E_s x_T^{\NCal_s} - d_s\|\|\bar{\mu}^s\|)} \\
& \leq & \left(\max_{1 \leq s \leq S} \|\bar{\lambda}^s\|\right) \sum_{s=1}^S \|\max\{0, A_s x_T^{\NCal_s} - b_s\}\| + \left(\max_{1 \leq s \leq S} \|\bar{\mu}^s\|\right) \sum_{s=1}^S \|E_s x_T^{\NCal_s} - d_s\| \nonumber \\ 
& \leq & \left(\max_{1 \leq s \leq S} \sqrt{\tfrac{2\|\bar{\lambda}^s\|^2 S}{\beta_T^s}}\right) \sqrt{g_1(x_T)} + \left(\max_{1 \leq s \leq S} \sqrt{\tfrac{2\|\bar{\mu}^s\|^2 S}{\rho_T^s}}\right) \sqrt{h_1(x_T)}. \nonumber
\end{eqnarray}
\paragraph{Proof of Eq.~\eqref{opt:eps-QP-feas}.} Combining Eq.~\eqref{inequality:main-QP-key} with Eq.~\eqref{inequality:main-QP-fourth}, we have
\begin{equation*}
g_1(x_T) + h_1(x_T) \leq \delta_T + \left(\sqrt{2S}\max_{1 \leq s \leq S} \left\{\sqrt{\tfrac{\|\bar{\lambda}^s\|^2}{\beta_T^s}} + \sqrt{\tfrac{\|\bar{\mu}^s\|^2}{\rho_T^s}}\right\}\right) \sqrt{g_1(x_T) + h_1(x_T)}. 
\end{equation*}
After some simple calculations, we have
\begin{equation*}
g_1(x_T) + h_1(x_T) \leq 2\delta_T + 20S \left(\max_{1 \leq s \leq S} \left\{\tfrac{\|\bar{\lambda}^s\|^2}{\beta_T^s} + \tfrac{\|\bar{\mu}^s\|^2}{\rho_T^s}\right\}\right). 
\end{equation*}
Then, Eq.~\eqref{opt:eps-QP-feas} follows from the definition of $g_1$, $h_1$ and $\hat{\epsilon}_1$. 
\paragraph{Proof of Eq.~\eqref{opt:eps-QP-feas-VI}.} It follows from $g_1(x_T) \geq 0$, $h_1(x_T) \geq 0$ and Eq.~\eqref{inequality:main-QP-third} that $- g_1(x) - h_1(x) + (x_T - x)^\top v(x) \leq \delta_T$ for all $x \in \hat{X}$. Fixing $\nu \in \NCal$, we plug $x^{-\nu} = x_T^{-\nu}$ into the above inequality and obtain that 
\begin{eqnarray*}
\lefteqn{(x_T^\nu - x^\nu)^\top v_\nu(x^\nu, x_T^{-\nu}) - \sum_{s \in \ICal_\nu} \tfrac{\beta_T^s}{2}\|\max\{0, A_s^\nu x^\nu + \sum_{j \in \NCal_s, j \neq \nu} A_s^j x_T^j - b_s\}\|^2} \\ 
& & - \sum_{s \in \ICal_\nu} \tfrac{\rho_T^s}{2}\|E_s^\nu x^\nu + \sum_{j \in \NCal_s, j \neq \nu} E_s^j x_T^j - d_s\|^2 \leq \delta_T, \quad \textnormal{for all } x^\nu \in \hat{X}_\nu. 
\end{eqnarray*}
Suppose that $x^\nu \in \hat{X}_\nu$ satisfies Eq.~\eqref{opt:eps-QP-feas-VI}, we obtain from $\ICal_\nu \subseteq [S]$ that  
\begin{equation*}
(x_T^\nu - x^\nu)^\top v_\nu(x^\nu, x_T^{-\nu}) \leq \delta_T + \tfrac{\epsilon^2}{2}\left(\sum_{s = 1}^S \beta_T^s + \rho_T^s\right). 
\end{equation*}
Then, Eq.~\eqref{opt:eps-QP-feas-VI} follows from the definition of $\hat{\epsilon}_2$. 
\end{proof}

\subsection{Proof of Theorem~\ref{Thm:AMPQP}} 
Fixing a sufficiently small $\epsilon \in (0, 1)$, we have $x_T \in \hat{X}$ is an $\epsilon$-solution of an NGNEP if
\begin{equation*}
\|\max\{0, A_s x_T^{\NCal_s} - b_s\}\| \leq \epsilon, \quad \|E_s x_T^{\NCal_s} - d_s\| \leq \epsilon, \quad \textnormal{for all } s \in [S], 
\end{equation*}
and, for all $\nu \in \NCal$, we have 
\begin{equation*}
(x_T^\nu - x^\nu)^\top v_\nu(x^\nu, \bar{x}^{-\nu}) \leq C\epsilon, 
\end{equation*}
for all $x^\nu \in \hat{X}_\nu$ satisfying that 
\begin{equation*}
\|\max\{0, A_s^\nu x^\nu + \sum_{j \in \NCal_s, j \neq \nu} A_s^j x_T^j - b_s\}\| \leq \epsilon, \quad \|E_s^\nu x^\nu + \sum_{j \in \NCal_s, j \neq \nu} E_s^j x_T^j - d_s\| \leq \epsilon, \quad \textnormal{for all } s \in \ICal_\nu, 
\end{equation*}
By Lemma~\ref{Lemma:main-QP}, it suffices to guarantee that $\delta_T > 0$ and $(\beta_T^s, \rho_T^s)_{s \in [S]}$ satisfy the following conditions:   
\begin{equation*}
\begin{array}{rcl}
\tfrac{4\delta_T + 40S\left(\max_{1 \leq s \leq S} \left\{\tfrac{\|\bar{\lambda}^s\|^2}{\beta_T^s} + \tfrac{\|\bar{\mu}^s\|^2}{\rho_T^s}\right\}\right)}{\min_{1 \leq s \leq S} \{\beta_T^s, \rho_T^s\}} & \leq & \epsilon^2, \\
\delta_T + \tfrac{\epsilon^2}{2}\left(\sum_{s = 1}^S \beta_T^s + \rho_T^s\right) & \leq & C\epsilon. 
\end{array}
\end{equation*}
From the update rule in Algorithm~\ref{alg:AMPQP}, we have $\beta_T^s = \gamma^T \beta_0^s$ and $\rho_T^s = \gamma^T \rho_0^s$ for all $s \in [S]$ and $\delta_T = \tfrac{\delta_0}{\gamma^T}$ in which $\gamma > 1$. Putting these pieces together, it suffices to guarantee that 
\begin{equation}\label{inequality:AMPQP-main}
\begin{array}{rcl}
\tfrac{\epsilon}{2} & \geq & \tfrac{1}{\gamma^T}\sqrt{\tfrac{\delta_0 + 10S \left(\max_{1 \leq s \leq S} \left\{\tfrac{\|\bar{\lambda}^s\|^2}{\beta_0^s} + \tfrac{\|\bar{\mu}^s\|^2}{\rho_0^s}\right\}\right)}{\min_{1 \leq s \leq S} \{\beta_0^s, \rho_0^s\}}}, \\
C\epsilon & \geq & \tfrac{\delta_0}{\gamma^T} + \tfrac{\gamma^T\epsilon^2}{2} \left(\sum_{s = 1}^S \beta_0^s + \rho_0^s\right). 
\end{array}
\end{equation}
Suppose that we set $T > 0$ as
\begin{equation*}
T = 1 + \left\lfloor \tfrac{\log\left(\tfrac{2}{\epsilon}\right) + \tfrac{1}{2}\log\left(\tfrac{\delta_0 + 10S\left(\max_{1 \leq s \leq S} \left\{\tfrac{\|\bar{\lambda}^s\|^2}{\beta_0^s} + \tfrac{\|\bar{\mu}^s\|^2}{\rho_0^s}\right\}\right)}{\min_{1 \leq s \leq S} \{\beta_0^s, \rho_0^s\}}\right)}{\log(\gamma)}\right\rfloor.   
\end{equation*}
Then, Eq.~\eqref{inequality:AMPQP-main} holds true with a positive constant $C > 0$ given by 
\begin{eqnarray*}
C & = & 2 + \left\lceil\left(\sum_{s = 1}^S \beta_0^s + \rho_0^s\right)\left(\tfrac{\delta_0 + 10S\left(\max_{1 \leq s \leq S} \left\{\tfrac{\|\bar{\lambda}^s\|^2}{\beta_0^s} + \tfrac{\|\bar{\mu}^s\|^2}{\rho_0^s}\right\}\right)}{\min_{1 \leq s \leq S} \{\beta_0^s, \rho_0^s\}}\right)^{1/2} \right. \\
& & \left. + \tfrac{\delta_0}{2}\left(\tfrac{\delta_0 + 10S\left(\max_{1 \leq s \leq S} \left\{\tfrac{\|\bar{\lambda}^s\|^2}{\beta_0^s} + \tfrac{\|\bar{\mu}^s\|^2}{\rho_0^s}\right\}\right)}{\min_{1 \leq s \leq S} \{\beta_0^s, \rho_0^s\}}\right)^{-1/2}\right\rceil.  
\end{eqnarray*}
Since $\beta_k^s = \gamma^k \beta_0^s$ and $\rho_k^s = \gamma^k \rho_0^s$ for all $s \in [S]$ and $\delta_k = \tfrac{\delta_0}{\gamma^k}$, Lemma~\ref{Lemma:subroutine-QP} guarantees that the required number of gradient evaluations at the $k^\textnormal{th}$ iteration is bounded by
\begin{equation*}
N_k = \left\{
\begin{array}{ll}
O\left(\gamma^k\left(\sqrt{\tfrac{\sum_{s=1}^S (\beta_0^s\|A_s\|^2 + \rho_0^s\|E_s\|^2) D^2}{\delta_0}} + \tfrac{\sqrt{N}\ell_\theta D^2}{\delta_0}\right)\right), & \textnormal{if } \alpha=0, \\
O\left(\left(\gamma^{\tfrac{k}{2}}\sqrt{\tfrac{\sum_{s=1}^S (\beta_0^s\|A_s\|^2 + \rho_0^s\|E_s\|^2)}{\alpha}} + \tfrac{\sqrt{N}\ell_\theta}{\alpha}\right)\log\left(\tfrac{\gamma^k(\sqrt{N}\ell_\theta+\ell_G)D^2}{\delta_0}\right)\right), & \textnormal{if } \alpha>0. \\
\end{array}
\right. 
\end{equation*}
Therefore, we conclude that the total number of gradient evaluations required to return an $\epsilon$-solution of the NGNEP is
\begin{equation*}
N_{\textnormal{grad}} = \sum_{k=1}^{T-1} N_k = \left\{
\begin{array}{ll}
O\left(\epsilon^{-1}\right), & \textnormal{if } \alpha = 0, \\
O\left(\epsilon^{-1/2}\log(1/\epsilon)\right), & \textnormal{if } \alpha > 0. \\
\end{array}
\right. 
\end{equation*}
This completes the proof. 

\section{Iteration Complexity Analysis of the AMPAL Algorithm}\label{app:AMPAL-proof}
We provide the proof of Theorem~\ref{Thm:AMPAL} on the iteration complexity bound of the AMPAL algorithm (cf. Algorithm~\ref{alg:AMPAL}) for solving monotone and strongly monotone NGNEPs.

\subsection{Technical lemmas}
We present some technical lemmas which are important to the subsequent analysis. For the ease of presentation, we first review the quadratic penalization of NGNEPs based on the augmented Lagrangian function. In particular, we consider the augmented Lagrangian function for each player's minimization problem given by 
\begin{equation*}
\LCal_\nu(x, \lambda, \mu) = \theta_\nu(x) + \left(\sum_{s \in \ICal_\nu} \tfrac{\beta^s}{2}\|\max\{0, A_s x^{\NCal_s}-b_s+\tfrac{\lambda^s}{\beta^s}\}\|^2 + \tfrac{\rho^s}{2}\|E_s x^{\NCal_s}-d_s+\tfrac{\mu^s}{\rho^s}\|^2\right),  
\end{equation*}
where $(\beta^s, \rho^s) \in \br_+ \times \br_+$ and $(\lambda^s, \mu^s) \in \br^{m_s} \times \br^{e_s}$ stand for penalty parameters and Lagrangian multipliers associated with inequality and equality constraints. We define the functions $g_2: \br^n \mapsto \br$ and $h_2: \br^n \mapsto \br$ by
\begin{equation*}
g_2(x) = \sum_{s=1}^S \tfrac{\beta^s}{2}\|\max\{0, A_s x^{\NCal_s}-b_s+\tfrac{\lambda^s}{\beta^s}\}\|^2, \quad h_2(x) = \sum_{s=1}^S \tfrac{\rho^s}{2}\|E_s x^{\NCal_s}-d_s+\tfrac{\mu^s}{\rho^s}\|^2. 
\end{equation*}
Concatenating the first-order optimality conditions of minimizing the augmented Lagrangian function for each player, we aim at solving the following VI ($\hat{X}$ is convex and compact with a diameter $D>0$): 
\begin{equation}\label{app:PVI-AL}
\textnormal{Find } x \in \hat{X}: \quad (x' - x)^\top(v(x) + \nabla g_2(x) + \nabla h_2(x)) \geq 0, \  \textnormal{ for all } x' \in \hat{X}.     
\end{equation}
As an immediate consequence of Proposition~\ref{Prop:smooth}, we obtain the following lemma. We omit the proof since it is the same as that of Lemma~\ref{Lemma:smooth-QP}.  
\begin{lemma}
The gradients of the functions $g_2$ and $h_2$ are also $\ell_\beta$-Lipschitz continuous and $\ell_\rho$-Lipschitz continuous respectively, where $\ell_\beta \mydefn \sum_{s=1}^S \beta^s\|A_s\|^2$ and $\ell_\rho \mydefn \sum_{s=1}^S \rho^s\|E_s\|^2$. 
\end{lemma}
Combining this lemma with Theorem~\ref{Thm:AMP-MT} and~\ref{Thm:AMP-SMT}, we get the following lemma for the subroutine in Algorithm~\ref{alg:AMPAL}. The proof is omitted since it is the same as that of Lemma~\ref{Lemma:subroutine-QP}.   
\begin{lemma}\label{Lemma:subroutine-AL}
Suppose that $x_{k+1} = \textsc{amp}((\beta_{k+1}^s, \rho_{k+1}^s)_{s \in [S]}, (\lambda_k^s, \mu_k^s)_{s \in [S]}, \delta_{k+1}, x_k)$ in the sense that $x_{k+1}$ is a $\delta_{k+1}$-approximate solution to the structured VI in Eq.~\eqref{app:PVI-AL} with $(\beta^s, \rho^s) = (\beta_{k+1}^s, \rho_{k+1}^s)$ and $(\lambda^s, \mu^s) = (\lambda_k^s, \mu_k^s)$ such that 
\begin{equation*}
\max_{x \in \hat{X}} \left\{g_2(x_{k+1}) + h_2(x_{k+1}) - g_2(x) - h_2(x) + (x_{k+1} - x)^\top v(x)\right\} \leq \delta_{k+1}. 
\end{equation*}
The required number of gradient evaluations at the $k^\textnormal{th}$ iteration (for $k \geq 1$) is bounded by
\begin{equation*}
N_k = \left\{\begin{array}{ll}
O\left(\sqrt{\tfrac{\sum_{s=1}^S (\beta_k^s\|A_s\|^2 + \rho_k^s\|E_s\|^2) D^2}{\delta_k}} + \tfrac{\sqrt{N}\ell_\theta D^2}{\delta_k}\right), & \textnormal{if } \alpha = 0, \\
O\left(\left(\sqrt{\tfrac{\sum_{s=1}^S (\beta_k^s\|A_s\|^2 + \rho_k^s\|E_s\|^2)}{\alpha}} + \tfrac{\sqrt{N}\ell_\theta}{\alpha}\right)\log\left(\tfrac{(\sqrt{N}\ell_\theta+\ell_G)D^2}{\delta_k}\right)\right), & \textnormal{if } \alpha > 0. \\
\end{array}\right. 
\end{equation*}
\end{lemma}
To establish the convergence of the AMPAL algorithm, we provide the following lemma. 
\begin{lemma}\label{Lemma:dual-AL}
Suppose that $\{(\lambda_k^s, \mu_k^s)_{s \in [S]}\}_{k \geq 1}$ is generated by the AMPAL algorithm. Then, for any $(\lambda^s, \mu^s) \in \br_+^{m_s} \times \br^{e_s}$, we have
\begin{equation*}
\begin{array}{rcl}
\tfrac{1}{2\beta_{k+1}^s}(\|\lambda_{k+1}^s - \lambda^s\|^2 - \|\lambda_k^s - \lambda^s\|^2 + \|\lambda_{k+1}^s - \lambda_k^s\|^2) & \leq & (\lambda_{k+1}^s - \lambda^s)^\top(A_s x_{k+1}^{\NCal_s} - b_s), \\
\frac{1}{2\rho_{k+1}^s}(\|\mu_{k+1}^s - \mu^s\|^2 - \|\mu_k^s - \mu^s\|^2 + \|\mu_{k+1}^s - \mu_k^s\|^2) & = & (\mu_{k+1}^s - \mu^s)^\top(E_s x_{k+1}^{\NCal_s} - d_s).    
\end{array}
\end{equation*}
\end{lemma}
\begin{proof}
By the update of $\lambda_k^s$ and $\mu_k^s$ in Algorithm~\ref{alg:AMPAL}, we have 
\begin{eqnarray*}
\lefteqn{\tfrac{1}{2\beta_{k+1}^s}(\|\lambda_{k+1}^s - \lambda^s\|^2 - \|\lambda_k^s - \lambda^s\|^2 + \|\lambda_{k+1}^s - \lambda_k^s\|^2)} \\
& = & \tfrac{1}{\beta_{k+1}^s}(\lambda_{k+1}^s - \lambda^s)^\top\max\{-\lambda_k^s, \beta_{k+1}^s(A_s x_{k+1}^{\NCal_s} - b_s)\}, 
\end{eqnarray*}
and 
\begin{equation*}
\tfrac{1}{2\rho_{k+1}^s}(\|\mu_{k+1}^s - \mu^s\|^2 - \|\mu_k^s - \mu^s\|^2 + \|\mu_{k+1}^s - \mu_k^s\|^2) = (\mu_{k+1}^s - \mu^s)^\top(E_s x_{k+1}^{\NCal_s} - d_s), 
\end{equation*}
where the second equality is one of desired results. Let us denote $I_k^s$ as the set of coordinates of $\lambda_k^s + \beta_{k+1}^s(A_s x_{k+1}^{\NCal_s} - b_s) \in \br^{m^s}$ whose values are nonnegative. This together with the update of $\lambda_{k+1}^s$ yields that $(\lambda_{k+1}^s)_j = 0$ for all $j \in [m^s] \setminus I_k^s$. Then, we have
\begin{eqnarray*}
\lefteqn{\tfrac{1}{\beta_{k+1}^s}(\lambda_{k+1}^s - \lambda^s)^\top\max\{-\lambda_k^s, \beta_{k+1}^s(A_s x_{k+1}^{\NCal_s} - b_s)\}} \\ 
& = & \sum_{j \in I_k^s} (\lambda_{k+1}^s - \lambda^s)_j(A_s x_{k+1}^{\NCal_s} - b_s)_j + \sum_{j \in [m^s] \setminus I_k^s} (- \lambda^s)_j\left(-\tfrac{\lambda_k^s}{\beta_{k+1}^s}\right)_j. 
\end{eqnarray*}
Since $\lambda^s \geq 0$ and $\left(-\tfrac{\lambda_k^s}{\beta_{k+1}^s}\right)_j > (A_s x_{k+1}^{\NCal_s} - b_s)_j$ for all $j \in [m^s] \setminus I_k^s$, we have
\begin{equation*}
\sum_{j \in [m^s] \setminus I_k^s} (- \lambda^s)_j\left(-\tfrac{\lambda_k^s}{\beta_{k+1}^s}\right)_j \leq \sum_{j \in [m^s] \setminus I_k^s} (- \lambda^s)_j(A_s x_{k+1}^{\NCal_s} - b_s)_j = \sum_{j \in [m^s] \setminus I_k^s} (\lambda_{k+1}^s - \lambda^s)_j(A_s x_{k+1}^{\NCal_s} - b_s)_j. 
\end{equation*}
Putting these pieces together yields that 
\begin{equation*}
\tfrac{1}{\beta_{k+1}^s}(\lambda_{k+1}^s - \lambda^s)^\top\max\{-\lambda_k^s, \beta_{k+1}^s(A_s x_{k+1}^{\NCal_s} - b_s)\} \leq (\lambda_{k+1}^s - \lambda^s)^\top (A_s x_{k+1}^{\NCal_s} - b_s). 
\end{equation*}
This completes the proof. 
\end{proof}
We continue with studying the per-iteration progress of the AMPAL algorithm. 
\begin{lemma}\label{Lemma:main-AL}
Suppose that $x_{k+1} = \textsc{amp}((\beta_{k+1}^s, \rho_{k+1}^s)_{s \in [S]}, (\lambda_k^s, \mu_k^s)_{s \in [S]}, \delta_{k+1}, x_k)$ in the sense that $x_{k+1}$ is a $\delta_{k+1}$-approximate solution to the structured VI in Eq.~\eqref{app:PVI-AL} with $(\beta^s, \rho^s) = (\beta_{k+1}^s, \rho_{k+1}^s)$ and $(\lambda^s, \mu^s) = (\lambda_k^s, \mu_k^s)$ such that 
\begin{equation*}
\max_{x \in \hat{X}} \left\{g_2(x_{k+1}) + h_2(x_{k+1}) - g_2(x) - h_2(x) + (x_{k+1} - x)^\top v(x)\right\} \leq \delta_{k+1}. 
\end{equation*}
Then,  $x_T \in \hat{X}$ satisfies that 
\begin{equation}\label{opt:eps-AL-feas}
\|\max\{0, A_s x_T^{\NCal_s} - b_s\}\| \leq \hat{\epsilon}_1, \quad \|E_s x_T^{\NCal_s} - d_s\| \leq \hat{\epsilon}_1, \quad \textnormal{for all } s \in [S], 
\end{equation}
where $\hat{\epsilon}_1 > 0$ is defined by 
\begin{equation*}
\hat{\epsilon}_1 = \tfrac{1}{\gamma^T}\left(\delta_0 T + \left(\sum_{s=1}^S \tfrac{\|\mu_0^s - \bar{\mu}^s\|^2}{2\rho_0^s} + \tfrac{\|\lambda_0^s - \bar{\lambda}^s\|^2}{2\beta_0^s}\right)\right), 
\end{equation*}
and for all $\nu \in \NCal$, we have 
\begin{equation*}
(x_T^\nu - x^\nu)^\top v_\nu(x^\nu, x_T^{-\nu}) \leq \hat{\epsilon}_2, 
\end{equation*}
for all $x^\nu \in \hat{X}_\nu$ satisfying that 
\begin{equation}\label{opt:eps-AL-feas-VI}
\|\max\{0, A_s^\nu x^\nu + \sum_{j \in \NCal_s, j \neq \nu} A_s^j x_T^j - b_s\}\| \leq \epsilon, \quad \|E_s^\nu x^\nu + \sum_{j \in \NCal_s, j \neq \nu} E_s^j x_T^j - d_s\| \leq \epsilon, \quad \textnormal{for all } s \in \ICal_\nu, 
\end{equation}
where $\hat{\epsilon}_2 > 0$ is defined by 
\begin{equation*}
\hat{\epsilon}_2 = \hat{\epsilon}_1 + \tfrac{\epsilon^2(1+\gamma^T)}{2} \left(\sum_{s = 1}^S \beta_0^s + \rho_0^s\right) + \epsilon^2\left(2\delta_0(T-1) + \left(\sum_{s=1}^S \tfrac{\|\mu_0^s - \bar{\mu}^s\|^2 + \|\bar{\mu}^s\|^2}{\rho_0^s} + \tfrac{\|\lambda_0^s - \bar{\lambda}^s\|^2 + \|\bar{\lambda}^s\|^2}{\beta_0^s}\right)\right). 
\end{equation*} 
\end{lemma}
\begin{proof}
Since $x_{k+1} = \textsc{amp}((\beta_{k+1}^s, \rho_{k+1}^s)_{s \in [S]}, (\lambda_k^s, \mu_k^s)_{s \in [S]}, \delta_{k+1}, x_k)$, we have
\begin{eqnarray}\label{inequality:main-AL-first}
\lefteqn{\delta_{k+1} \geq g_2(x_{k+1}) + h_2(x_{k+1}) - g_2(x) - h_2(x) + (x_{k+1} - x)^\top v(x)} \\
& = & \left(\sum_{s=1}^S \left(\tfrac{\rho_{k+1}^s}{2}\|E_s x_{k+1}^{\NCal_s}-d_s+\tfrac{\mu_k^s}{\rho_{k+1}^s}\|^2 - \tfrac{\rho_{k+1}^s}{2}\|E_s x^{\NCal_s}-d_s+\tfrac{\mu_k^s}{\rho_{k+1}^s}\|^2\right)\right) \nonumber \\
& & + \left(\sum_{s=1}^S \tfrac{\beta_{k+1}^s}{2}\|\max\{0, A_s x_{k+1}^{\NCal_s}-b_s+\tfrac{\lambda_k^s}{\beta_{k+1}^s}\}\|^2\right) \nonumber \\ 
& & - \left(\sum_{s=1}^S \tfrac{\beta_{k+1}^s}{2}\|\max\{0, A_s x^{\NCal_s}-b_s+\tfrac{\lambda_k^s}{\beta_{k+1}^s}\}\|^2\right) \nonumber \\ 
& & + (x_{k+1} - x)^\top v(x), \textnormal{ for all } x \in \hat{X}. \nonumber  
\end{eqnarray}
Using the update of $\mu_k^s$ in Algorithm~\ref{alg:AMPAL} together with the second equality in Lemma~\ref{Lemma:dual-AL}, we have
\begin{eqnarray}\label{inequality:main-AL-second}
\lefteqn{\tfrac{\rho_{k+1}^s}{2}\|E_s x_{k+1}^{\NCal_s}-d_s+\tfrac{\mu_k^s}{\rho_{k+1}^s}\|^2} \\
& = & (\mu_k^s)^\top(E_s x_{k+1}^{\NCal_s}-d_s) + \tfrac{\rho_{k+1}^s}{2}\|E_s x_{k+1}^{\NCal_s}-d_s\|^2 + \tfrac{\|\mu_k^s\|^2}{2\rho_{k+1}^s} \nonumber \\ 
& = & (\mu^s)^\top(E_s x_{k+1}^{\NCal_s}-d_s) + (\mu_{k+1}^s - \mu^s)^\top(E_s x_{k+1}^{\NCal_s}-d_s) - \tfrac{1}{2\rho_{k+1}^s} \|\mu_{k+1}^s - \mu_k^s\|^2 + \tfrac{\|\mu_k^s\|^2}{2\rho_{k+1}^s} \nonumber \\ 
& = & (\mu^s)^\top(E_s x_{k+1}^{\NCal_s}-d_s) + \tfrac{1}{2\rho_{k+1}^s}(\|\mu_{k+1}^s - \mu^s\|^2 - \|\mu_k^s - \mu^s\|^2 + \|\mu_k^s\|^2). \nonumber
\end{eqnarray}
Recall that $I_k^s$ denotes the set of coordinates of $\lambda_k^s + \beta_{k+1}^s(A_s x_{k+1}^{\NCal_s} - b_s)$ whose values are nonnegative. Then, by the update of $\lambda_k^s$ in Algorithm~\ref{alg:AMPAL}, we have
\begin{eqnarray*}
\lefteqn{\tfrac{\beta_{k+1}^s}{2}\|\max\{0, A_s x_{k+1}^{\NCal_s}-b_s+\tfrac{\lambda_k^s}{\beta_{k+1}^s}\}\|^2 = \sum_{j \in I_k^s} \tfrac{\beta_{k+1}^s}{2}\|(A_s x_{k+1}^{\NCal_s}-b_s+\tfrac{\lambda_k^s}{\beta_{k+1}^s})_j\|^2} \\
& = & \sum_{j \in I_k^s} \left((\lambda_k^s)_j(A_s x_{k+1}^{\NCal_s}-b_s)_j + \tfrac{\beta_{k+1}^s}{2}((A_s x_{k+1}^{\NCal_s}-b_s)_j)^2 + \tfrac{((\lambda_k^s)_j)^2}{2\beta_{k+1}^s}\right) \\
& = & \sum_{j \in I_k^s} \left((\lambda_{k+1}^s)_j(A_s x_{k+1}^{\NCal_s}-b_s)_j - \tfrac{\beta_{k+1}^s}{2}((A_s x_{k+1}^{\NCal_s}-b_s)_j)^2 + \tfrac{((\lambda_k^s)_j)^2}{2\beta_{k+1}^s}\right) \\
& = & (\lambda_{k+1}^s)^\top(A_s x_{k+1}^{\NCal_s}-b_s) - \sum_{j \in I_k^s} \tfrac{\beta_{k+1}^s}{2}((A_s x_{k+1}^{\NCal_s}-b_s)_j)^2 + \sum_{j \in I_k^s} \tfrac{((\lambda_k^s)_j)^2}{2\beta_{k+1}^s}, 
\end{eqnarray*}
and 
\begin{equation*}
\tfrac{1}{2\beta_{k+1}^s}\|\lambda_{k+1}^s - \lambda_k^s\|^2 = \sum_{j \in I_k^s} \tfrac{\beta_{k+1}^s}{2}((A_s x_{k+1}^{\NCal_s}-b_s)_j)^2 + \sum_{j \in [m^s] \setminus I_k^s} \tfrac{((\lambda_k^s)_j)^2}{2\beta_{k+1}^s}. 
\end{equation*}
Putting these two equations together yields that
\begin{equation}\label{inequality:main-AL-third}\small
\tfrac{\beta_{k+1}^s}{2}\|\max\{0, A_s x_{k+1}^{\NCal_s}-b_s+\tfrac{\lambda_k^s}{\beta_{k+1}^s}\}\|^2 =-\tfrac{1}{2\beta_{k+1}^s}\|\lambda_{k+1}^s - \lambda_k^s\|^2 + (\lambda_{k+1}^s)^\top(A_s x_{k+1}^{\NCal_s}-b_s) + \tfrac{\|\lambda_k^s\|^2}{2\beta_{k+1}^s}. 
\end{equation}
Combining Eq.~\eqref{inequality:main-AL-third} with the first inequality in Lemma~\ref{Lemma:dual-AL}, we have
\begin{equation}\label{inequality:main-AL-fourth}\small
\tfrac{\beta_{k+1}^s}{2}\|\max\{0, A_s x_{k+1}^{\NCal_s}-b_s+\tfrac{\lambda_k^s}{\beta_{k+1}^s}\}\|^2 \geq (\lambda^s)^\top(A_s x_{k+1}^{\NCal_s}-b_s) + \tfrac{\|\lambda_k^s\|^2}{2\beta_{k+1}^s} + \tfrac{\|\lambda_{k+1}^s - \lambda^s\|^2 - \|\lambda_k^s - \lambda^s\|^2}{2\beta_{k+1}^s}. 
\end{equation}
Plugging Eq.~\eqref{inequality:main-AL-second} and Eq.~\eqref{inequality:main-AL-fourth} into Eq.~\eqref{inequality:main-AL-first} yields that, for any $(\lambda^s, \mu^s) \in \br_+^{m_s} \times \br^{e_s}$, we have
\begin{eqnarray*}
\lefteqn{\delta_{k+1} \geq \left(\sum_{s=1}^S \tfrac{\|\mu_{k+1}^s - \mu^s\|^2 - \|\mu_k^s - \mu^s\|^2}{2\rho_{k+1}^s} + \tfrac{\|\lambda_{k+1}^s - \lambda^s\|^2 - \|\lambda_k^s - \lambda^s\|^2}{2\beta_{k+1}^s} + \tfrac{\|\mu_k^s\|^2}{2\rho_{k+1}^s} + \tfrac{\|\lambda_k^s\|^2}{2\beta_{k+1}^s}\right)} \\
& & + \left(\sum_{s=1}^S (\mu^s)^\top(E_s x_{k+1}^{\NCal_s}-d_s) + (\lambda^s)^\top(A_s x_{k+1}^{\NCal_s}-b_s)\right) \\
& & - \left(\sum_{s=1}^S \tfrac{\rho_{k+1}^s}{2}\|E_s x^{\NCal_s}-d_s + \tfrac{\mu_k^s}{\rho_{k+1}^s}\|^2 + \tfrac{\beta_{k+1}^s}{2}\|\max\{0, A_s x^{\NCal_s}-b_s+\tfrac{\lambda_k^s}{\beta_{k+1}^s}\}\|^2\right) \\
& & + (x_{k+1} - x)^\top v(x), \textnormal{ for all } x \in \hat{X}.
\end{eqnarray*}
Rearranging this inequality with the update of $\delta_k$, $\rho_k^s$ and $\beta_k^s$, we have
\begin{eqnarray}\label{inequality:main-AL-key}
\lefteqn{\delta_0 \geq \left(\sum_{s=1}^S \tfrac{\|\mu_{k+1}^s - \mu^s\|^2 - \|\mu_k^s - \mu^s\|^2}{2\rho_0^s} + \tfrac{\|\lambda_{k+1}^s - \lambda^s\|^2 - \|\lambda_k^s - \lambda^s\|^2}{2\beta_0^s} + \tfrac{\|\mu_k^s\|^2}{2\rho_0^s} + \tfrac{\|\lambda_k^s\|^2}{2\beta_0^s}\right)} \\
& & + \gamma^{k+1}\left(\sum_{s=1}^S (\mu^s)^\top(E_s x_{k+1}^{\NCal_s}-d_s) + (\lambda^s)^\top(A_s x_{k+1}^{\NCal_s}-b_s)\right) \nonumber \\
& & - \left(\sum_{s=1}^S \tfrac{\rho_0^s\gamma^{2k+2}}{2}\|E_s x^{\NCal_s}-d_s + \tfrac{\mu_k^s}{\rho_{k+1}^s}\|^2 + \tfrac{\beta_0^s\gamma^{2k+2}}{2}\|\max\{0, A_s x^{\NCal_s}-b_s+\tfrac{\lambda_k^s}{\beta_{k+1}^s}\}\|^2\right) \nonumber \\
& & + \gamma^{k+1} (x_{k+1} - x)^\top v(x). \nonumber
\end{eqnarray}
By Proposition~\ref{Prop:VNE}, there exists a triple $(\bar{x} \in \hat{X}, \bar{\lambda} \geq 0, \bar{\mu}) $ such that  
\begin{equation}\label{condition:main-AL-first}
A_s \bar{x}^{\NCal_s} \leq b_s, \ E_s \bar{x}^{\NCal_s} = d_s, \ \langle \bar{\lambda}^s, A_s\bar{x}^{\NCal_s} - b_s\rangle = 0, \ \textnormal{for all } s \in [S], 
\end{equation}
and the following VI holds true for all $\nu \in \NCal$, 
\begin{equation}\label{condition:main-AL-second}
(\bar{x}^\nu - x^\nu)^\top\left(v_\nu(\bar{x}) + \sum_{s \in \ICal_\nu} ((A_s^\nu)^\top\bar{\lambda}^s + (E_s^\nu)^\top\bar{\mu}^s)\right) \leq 0, \quad \textnormal{for all } x^\nu \in \hat{X}_\nu. 
\end{equation}
Note that Eq.~\eqref{condition:main-AL-first} implies that 
\begin{equation}\label{condition:main-AL-third}
\begin{array}{rcl}
\tfrac{\rho_0^s\gamma^{2k+2}}{2}\|E_s\bar{x}^{\NCal_s}-d_s + \tfrac{ \mu_k^s}{\rho_0^s\gamma^{k+1}}\|^2 - \tfrac{\|\mu_k^s\|^2}{2\rho_0^s} & = & 0, \\
\tfrac{\beta_0^s\gamma^{2k+2}}{2}\|\max\{0, A_s\bar{x}^{\NCal_s}-b_s+\tfrac{\lambda_k^s}{\beta_0^s\gamma^{k+1}}\}\|^2 - \tfrac{\|\lambda_k^s\|^2}{2\beta_0^s} & \leq & 0. 
\end{array}
\end{equation}
\paragraph{Bounding Lagrangian multipliers.} Plugging Eq.~\eqref{condition:main-AL-third} into Eq.~\eqref{inequality:main-AL-key} with $x = \bar{x}$, we have
\begin{eqnarray}\label{inequality:main-AL-fifth}
\lefteqn{\delta_0 \geq \gamma^{k+1}(x_{k+1} - \bar{x})^\top v(\bar{x}) + \gamma^{k+1}\left(\sum_{s=1}^S (\mu^s)^\top(E_s x_{k+1}^{\NCal_s}-d_s) + (\lambda^s)^\top(A_s x_{k+1}^{\NCal_s}-b_s)\right)} \nonumber \\
& & + \left(\sum_{s=1}^S \tfrac{\|\mu_{k+1}^s - \mu^s\|^2 - \|\mu_k^s - \mu^s\|^2}{2\rho_0^s} + \tfrac{\|\lambda_{k+1}^s - \lambda^s\|^2 - \|\lambda_k^s - \lambda^s\|^2}{2\beta_0^s} + \tfrac{\|\mu_k^s\|^2}{2\rho_0^s} + \tfrac{\|\lambda_k^s\|^2}{2\beta_0^s}\right). 
\end{eqnarray}
Summing Eq.~\eqref{inequality:main-AL-fifth} over $k = 0, 1, \ldots, t-1$ and rearranging the resulting inequality with $\lambda^s = \bar{\lambda}^s$ and $\mu^s = \bar{\mu}^s$ for all $s \in [S]$ and $\widehat{x}_t = \tfrac{\sum_{k=1}^t \gamma^k x_k}{\sum_{k=1}^t \gamma^k}$, we have
\begin{eqnarray*}
\lefteqn{\tfrac{\delta_0 t}{\sum_{k=1}^t \gamma^k} \geq \left(\sum_{s=1}^S (\bar{\mu}^s)^\top(E_s\widehat{x}_t^{\NCal_s}-d_s) + (\bar{\lambda}^s)^\top(A_s\widehat{x}_t^{\NCal_s}-b_s)\right)} \\
& & + (\widehat{x}_t - \bar{x})^\top v(\bar{x}) + \tfrac{1}{\sum_{k=1}^t \gamma^k}\left(\sum_{s=1}^S \tfrac{\|\mu_t^s - \bar{\mu}^s\|^2 - \|\mu_0^s - \bar{\mu}^s\|^2}{2\rho_0^s} + \tfrac{\|\lambda_t^s - \bar{\lambda}^s\|^2 - \|\lambda_0^s - \bar{\lambda}^s\|^2}{2\beta_0^s} \right). \nonumber
\end{eqnarray*}
Applying a similar argument as in Lemma~\ref{Lemma:main-QP} together with Eq.~\eqref{condition:main-AL-second} and $\widehat{x}_t \in \hat{X}$, we have
\begin{equation*}
(\bar{x} - \widehat{x}_t)^\top v(\bar{x}) + \sum_{s=1}^S ((A_s\bar{x}^{\NCal_s} - A_s\widehat{x}_t^{\NCal_s})^\top \bar{\lambda}^s + (E_s\bar{x}^{\NCal_s} - E_s\widehat{x}_t^{\NCal_s})^\top \bar{\mu}^s) \leq 0.  
\end{equation*}
Summing the above two inequalities and using the fact that $E_s \bar{x}^{\NCal_s} = d_s$ and $(\bar{\lambda}^s)^\top(A_s\bar{x}^{\NCal_s} - b_s) = 0$ for all $s \in [S]$, we have 
\begin{equation*}
\delta_0 t \geq \sum_{s=1}^S \left(\tfrac{\|\mu_t^s - \bar{\mu}^s\|^2 - \|\mu_0^s - \bar{\mu}^s\|^2}{2\rho_0^s} + \tfrac{\|\lambda_t^s - \bar{\lambda}^s\|^2 - \|\lambda_0^s - \bar{\lambda}^s\|^2}{2\beta_0^s}\right).
\end{equation*}
Changing the index $t$ back to $k$ for the simplicity, we conclude that 
\begin{equation}\label{bound:main-AL-multiplier}
\sum_{s=1}^S \tfrac{\|\mu_k^s - \bar{\mu}^s\|^2}{2\rho_0^s} + \tfrac{\|\lambda_k^s - \bar{\lambda}^s\|^2}{2\beta_0^s} \leq \delta_0 k + \left(\sum_{s=1}^S \tfrac{\|\mu_0^s - \bar{\mu}^s\|^2}{2\rho_0^s} + \tfrac{\|\lambda_0^s - \bar{\lambda}^s\|^2}{2\beta_0^s}\right).  
\end{equation}
\paragraph{Proof of Eq.~\eqref{opt:eps-AL-feas}.} Considering Eq.~\eqref{inequality:main-AL-fifth} with $k = T-1$ and 
\begin{equation*}
\lambda^s = \bar{\lambda}^s + \tfrac{\max\{0, A_s x_T^{\NCal_s}-b_s\}}{\|\max\{0, A_s x_T^{\NCal_s}-b_s\}\|}, \quad \mu^s = \bar{\mu}^s + \tfrac{E_s x_T^{\NCal_s}-d_s}{\|E_s x_T^{\NCal_s}-d_s\|}, \quad \textnormal{for all } s \in [S]. 
\end{equation*}
Then, we have
\begin{eqnarray*}
\delta_0 & \geq & \gamma^T (x_T - \bar{x})^\top v(\bar{x}) - \left(\sum_{s=1}^S \tfrac{\left\|\mu_{T-1}^s - \bar{\mu}^s - \tfrac{E_s x_T^{\NCal_s}-d_s}{\|E_s x_T^{\NCal_s}-d_s\|}\right\|^2}{2\rho_0^s} + \tfrac{\left\|\lambda_{T-1}^s - \bar{\lambda}^s - \tfrac{\max\{0, A_s x_T^{\NCal_s}-b_s\}}{\|\max\{0, A_s x_T^{\NCal_s}-b_s\}\|}\right\|^2}{2\beta_0^s}\right) \\
& & + \gamma^T \left(\sum_{s=1}^S (\bar{\mu}^s)^\top(E_s x_T^{\NCal_s}-d_s) + (\bar{\lambda}^s)^\top(A_s x_T^{\NCal_s}-b_s)\right) \\
& & + \gamma^T \left(\sum_{s=1}^S \|E_s x_T^{\NCal_s}-d_s\| + \|\max\{0, A_s x_T^{\NCal_s}-b_s\}\|\right) \\
& \geq & \gamma^T (x_T - \bar{x})^\top v(\bar{x}) - \left(\sum_{s=1}^S \tfrac{\|\mu_{T-1}^s - \bar{\mu}^s\|^2 + 1}{\rho_0^s} + \tfrac{\|\lambda_{T-1}^s - \bar{\lambda}^s\|^2 + 1}{\beta_0^s}\right) \\
& & + \gamma^T \left(\sum_{s=1}^S (\bar{\mu}^s)^\top(E_s x_T^{\NCal_s}-d_s) + (\bar{\lambda}^s)^\top(A_s x_T^{\NCal_s}-b_s)\right) \\
& & + \gamma^T \left(\sum_{s=1}^S \|E_s x_T^{\NCal_s}-d_s\| + \|\max\{0, A_s x_T^{\NCal_s}-b_s\}\|\right). 
\end{eqnarray*}
Applying a similar argument as in Lemma~\ref{Lemma:main-QP} together with Eq.~\eqref{condition:main-AL-second} and $x_T \in \hat{X}$, we have
\begin{equation*}
(\bar{x} - x_T)^\top v(\bar{x}) + \sum_{s=1}^S ((A_s\bar{x}^{\NCal_s} - A_s x_T^{\NCal_s})^\top \bar{\lambda}^s + (E_s\bar{x}^{\NCal_s} - E_s x_T^{\NCal_s})^\top \bar{\mu}^s) \leq 0.  
\end{equation*}
Combining the above two inequalities and using that $E_s \bar{x}^{\NCal_s} = d_s$ and $(\bar{\lambda}^s)^\top(A_s\bar{x}^{\NCal_s} - b_s) = 0$ for all $s \in [S]$ and $\gamma^T > 0$, we have 
\begin{equation*}
\delta_0 \geq \gamma^T \left(\sum_{s=1}^S \|E_s x_T^{\NCal_s}-d_s\| + \|\max\{0, A_s x_T^{\NCal_s}-b_s\}\|\right) - \left(\sum_{s=1}^S \tfrac{\|\mu_{T-1}^s - \bar{\mu}^s\|^2 + 1}{\rho_0^s} + \tfrac{\|\lambda_{T-1}^s - \bar{\lambda}^s\|^2 + 1}{\beta_0^s}\right). 
\end{equation*}
This together with Eq.~\eqref{bound:main-AL-multiplier} yields that 
\begin{equation*}
\sum_{s=1}^S \|E_s x_T^{\NCal_s}-d_s\| + \|\max\{0, A_s x_T^{\NCal_s}-b_s\}\| \leq \tfrac{1}{\gamma^T}\left(\delta_0 T + \left(\sum_{s=1}^S \tfrac{\|\mu_0^s - \bar{\mu}^s\|^2}{2\rho_0^s} + \tfrac{\|\lambda_0^s - \bar{\lambda}^s\|^2}{2\beta_0^s}\right)\right). 
\end{equation*}
Then, Eq.~\eqref{opt:eps-AL-feas} follows from the definition of $\hat{\epsilon}_1$. 

\paragraph{Proof of Eq.~\eqref{opt:eps-AL-feas-VI}.} Considering Eq.~\eqref{inequality:main-AL-key} with $k = T-1$ and $\lambda^s = 0$ and $\mu^s = 0$ for all $s \in [S]$. Then, we have
\begin{eqnarray*}
\lefteqn{\delta_0 \geq \gamma^T(x_T - x)^\top v(x) - \left(\sum_{s=1}^S \tfrac{\|\mu_{T-1}^s\|^2}{2\rho_0^s} + \tfrac{\|\lambda_{T-1}^s\|^2}{2\beta_0^s}\right) + \left(\sum_{s=1}^S \tfrac{\|\mu_{T-1}^s\|^2}{2\rho_0^s} + \tfrac{\|\lambda_{T-1}^s\|^2}{2\beta_0^s}\right)} \\
& & - \left(\sum_{s=1}^S \tfrac{\rho_0^s\gamma^{2T}}{2}\|E_s x^{\NCal_s}-d_s + \tfrac{ \mu_{T-1}^s}{\rho_0^s\gamma^T}\|^2 + \tfrac{\beta_0^s\gamma^{2T}}{2}\|\max\{0, A_s x^{\NCal_s}-b_s+\tfrac{\lambda_{T-1}^s}{\beta_0^s\gamma^T}\}\|^2\right). 
\end{eqnarray*}
Rearranging this inequality and using Eq.~\eqref{bound:main-AL-multiplier} yields that 
\begin{eqnarray*}
\lefteqn{(x_T - x)^\top v(x) \leq \tfrac{1}{\gamma^T}\left(\delta_0 T + \left(\sum_{s=1}^S \tfrac{\|\mu_0^s - \bar{\mu}^s\|^2}{2\rho_0^s} + \tfrac{\|\lambda_0^s - \bar{\lambda}^s\|^2}{2\beta_0^s}\right)\right) - \left(\sum_{s=1}^S \tfrac{\|\mu_{T-1}^s\|^2}{2\rho_0^s\gamma^T} + \tfrac{\|\lambda_{T-1}^s\|^2}{2\beta_0^s\gamma^T}\right)} \\
& & + \left(\sum_{s=1}^S \tfrac{\rho_0^s\gamma^T}{2}\|E_s x^{\NCal_s}-d_s + \tfrac{\mu_{T-1}^s}{\rho_0^s\gamma^T}\|^2 + \tfrac{\beta_0^s\gamma^T}{2}\|\max\{0, A_s x^{\NCal_s}-b_s+\tfrac{\lambda_{T-1}^s}{\beta_0^s\gamma^T}\}\|^2\right). 
\end{eqnarray*}
Fixing $\nu \in \NCal$, we plug $x^{-\nu} = x_T^{-\nu}$ into this inequality and obtain that 
\begin{eqnarray*}
\lefteqn{(x_T^\nu - x^\nu)^\top v_\nu(x^\nu, x_T^{-\nu}) \leq \tfrac{1}{\gamma^T}\left(\delta_0 T + \left(\sum_{s=1}^S \tfrac{\|\mu_0^s - \bar{\mu}^s\|^2}{2\rho_0^s} + \tfrac{\|\lambda_0^s - \bar{\lambda}^s\|^2}{2\beta_0^s}\right)\right)} \\
& & + \underbrace{\sum_{s \in \ICal_\nu} \tfrac{\beta_0^s \gamma^T}{2}\|\max\{0, A_s^\nu x^\nu + \sum_{j \in \NCal_s, j \neq \nu} A_s^j x_T^j - b_s + \tfrac{\lambda_{T-1}^s}{\beta_0^s\gamma^T}\}\|^2 - \tfrac{\|\lambda_{T-1}^s\|^2}{2\beta_0^s\gamma^T}}_{\textbf{I}} \\ 
& & + \underbrace{\sum_{s \in \ICal_\nu} \tfrac{\rho_0^s\gamma^T}{2}\|E_s^\nu x^\nu + \sum_{j \in \NCal_s, j \neq \nu} E_s^j x_T^j - d_s + \tfrac{\mu_{T-1}^s}{\rho_0^s\gamma^T} \|^2 - \tfrac{\|\mu_{T-1}^s\|^2}{2\rho_0^s\gamma^T}}_{\textbf{II}}, \quad \textnormal{for all } x^\nu \in \hat{X}_\nu. 
\end{eqnarray*}
Suppose that $x^\nu \in \hat{X}_\nu$ satisfies  
\begin{equation*}
\|\max\{0, A_s^\nu x^\nu + \sum_{j \in \NCal_s, j \neq \nu} A_s^j x_T^j - b_s\}\| \leq \epsilon, \quad \|E_s^\nu x^\nu + \sum_{j \in \NCal_s, j \neq \nu} E_s^j x_T^j - d_s\| \leq \epsilon, \quad \textnormal{for all } s \in \ICal_\nu. 
\end{equation*}
Applying the above conditions to $\textbf{I}$ and $\textbf{II}$ and using Young's inequality, we have
\begin{eqnarray*}
\textbf{I} & \leq & \sum_{s \in \ICal_\nu} \tfrac{\beta_0^s \gamma^T}{2}\|\max\{0, A_s^\nu x^\nu + \sum_{j \in \NCal_s, j \neq \nu} A_s^j x_T^j - b_s\}\|^2 \\
& & + \sum_{s \in \ICal_\nu} \|\lambda_{T-1}^s\|\max\{0, A_s^\nu x^\nu + \sum_{j \in \NCal_s, j \neq \nu} A_s^j x_T^j - b_s\}\| \\
& \leq & \tfrac{\epsilon^2(1+\gamma^T)}{2} \left(\sum_{s \in \ICal_\nu} \beta_0^s\right) + \tfrac{\epsilon^2}{2}\left(\sum_{s \in \ICal_\nu} \tfrac{\|\lambda_{T-1}^s\|^2}{\beta_0^s}\right) \\
& \leq & \tfrac{\epsilon^2(1+\gamma^T)}{2} \left(\sum_{s \in \ICal_\nu} \beta_0^s\right) + \epsilon^2\left(\sum_{s \in \ICal_\nu} \tfrac{\|\lambda_{T-1}^s - \bar{\lambda}^s\|^2 + \|\bar{\lambda}^s\|^2}{\beta_0^s} \right), 
\end{eqnarray*}
and 
\begin{eqnarray*}
\textbf{II} & \leq & \sum_{s \in \ICal_\nu} \tfrac{\rho_0^s \gamma^T}{2}\|E_s^\nu x^\nu + \sum_{j \in \NCal_s, j \neq \nu} E_s^j x_T^j - d_s\|^2 + \|\mu_{T-1}^s\|E_s^\nu x^\nu + \sum_{j \in \NCal_s, j \neq \nu} E_s^j x_T^j - d_s\| \\
& \leq & \tfrac{\epsilon^2(1+\gamma^T)}{2} \left(\sum_{s \in \ICal_\nu} \rho_0^s\right) + \tfrac{\epsilon^2}{2}\left(\sum_{s \in \ICal_\nu} \tfrac{\|\mu_{T-1}^s\|^2}{\rho_0^s}\right) \\
& \leq & \tfrac{\epsilon^2(1+\gamma^T)}{2} \left(\sum_{s \in \ICal_\nu} \rho_0^s\right) + \epsilon^2\left(\sum_{s \in \ICal_\nu} \tfrac{\|\mu_{T-1}^s - \bar{\mu}^s\|^2 + \|\bar{\mu}^s\|^2}{\rho_0^s}\right). 
\end{eqnarray*}
Putting these pieces together with $\ICal_\nu \subseteq [S]$ and Eq.~\eqref{bound:main-AL-multiplier} yields that 
\begin{eqnarray*}
\lefteqn{(x_T^\nu - x^\nu)^\top(v_\nu(x^\nu, x_T^{-\nu})) \leq \tfrac{1}{\gamma^T}\left(\delta_0 T + \left(\sum_{s=1}^S \tfrac{\|\mu_0^s - \bar{\mu}^s\|^2}{2\rho_0^s} + \tfrac{\|\lambda_0^s - \bar{\lambda}^s\|^2}{2\beta_0^s}\right)\right)} \\
& & + \tfrac{\epsilon^2(1+\gamma^T)}{2} \left(\sum_{s = 1}^S \beta_0^s + \rho_0^s\right) + \epsilon^2\left(2\delta_0(T-1) + \left(\sum_{s=1}^S \tfrac{\|\mu_0^s - \bar{\mu}^s\|^2 + \|\bar{\mu}^s\|^2}{\rho_0^s} + \tfrac{\|\lambda_0^s - \bar{\lambda}^s\|^2 + \|\bar{\lambda}^s\|^2}{\beta_0^s}\right)\right), \\
& & \textnormal{for all } x^\nu \in \hat{X}_\nu. 
\end{eqnarray*}
Then, Eq.~\eqref{opt:eps-AL-feas-VI} follows from the definition of $\hat{\epsilon}_2$. 
\end{proof}
\subsection{Proof of Theorem~\ref{Thm:AMPAL}}
Fixing a sufficiently small $\epsilon \in (0, 1)$, we have $x_T \in \hat{X}$ is an $\epsilon$-solution of an NGNEP if
\begin{equation*}
\|\max\{0, A_s x_T^{\NCal_s} - b_s\}\| \leq \epsilon, \quad \|E_s x_T^{\NCal_s} - d_s\| \leq \epsilon, \quad \textnormal{for all } s \in [S], 
\end{equation*}
and for all $\nu \in \NCal$, we have 
\begin{equation*}
(x_T^\nu - x^\nu)^\top v_\nu(x^\nu, \bar{x}^{-\nu}) \leq C\epsilon, 
\end{equation*}
for all $x^\nu \in \hat{X}_\nu$ satisfying that 
\begin{equation*}
\|\max\{0, A_s^\nu x^\nu + \sum_{j \in \NCal_s, j \neq \nu} A_s^j x_T^j - b_s\}\| \leq \epsilon, \quad \|E_s^\nu x^\nu + \sum_{j \in \NCal_s, j \neq \nu} E_s^j x_T^j - d_s\| \leq \epsilon, \quad \textnormal{for all } s \in \ICal_\nu, 
\end{equation*}
By Lemma~\ref{Lemma:main-AL}, it suffices to guarantee that $T > 0$ satisfies the following conditions:   
\begin{equation}\label{inequality:AMPAL-main}
\begin{array}{rcl}
\epsilon & \geq & \tfrac{1}{\gamma^T}\left(\delta_0 T + \left(\sum_{s=1}^S \tfrac{\|\mu_0^s - \bar{\mu}^s\|^2}{2\rho_0^s} + \tfrac{\|\lambda_0^s - \bar{\lambda}^s\|^2}{2\beta_0^s}\right)\right), \\
C\epsilon & \geq & \epsilon^2\left(2\delta_0(T-1) + \left(\sum_{s=1}^S \tfrac{\|\mu_0^s - \bar{\mu}^s\|^2 + \|\bar{\mu}^s\|^2}{\rho_0^s} + \tfrac{\|\lambda_0^s - \bar{\lambda}^s\|^2 + \|\bar{\lambda}^s\|^2}{\beta_0^s}\right)\right) \\
& & + \hat{\epsilon}_1 + \tfrac{\epsilon^2(1+\gamma^T)}{2} \left(\sum_{s = 1}^S \beta_0^s + \rho_0^s\right). 
\end{array}
\end{equation}
Suppose that we set $T > 0$ as
\begin{equation*}
T = 1 + \left\lfloor \tfrac{\log\log\left(\tfrac{2+2\delta_0}{\epsilon}\right) + \log\left(\tfrac{2+2\delta_0}{\epsilon}\right) + \log\left(\sum_{s=1}^S \tfrac{\|\mu_0^s - \bar{\mu}^s\|^2}{\rho_0^s} + \tfrac{\|\lambda_0^s - \bar{\lambda}^s\|^2}{\beta_0^s}\right)}{\log(\gamma)}\right\rfloor.   
\end{equation*}
Then, Eq.~\eqref{inequality:AMPAL-main} holds true with a positive constant $C > 0$ given by 
\begin{eqnarray*}
\lefteqn{C = 1 + \left\lceil \left(\tfrac{1}{2} + (1 + \delta_0)\left(\sum_{s=1}^S \tfrac{\|\mu_0^s - \bar{\mu}^s\|^2}{\rho_0^s} + \tfrac{\|\lambda_0^s - \bar{\lambda}^s\|^2}{\beta_0^s}\right)\log\left(\tfrac{2+2\delta_0}{\epsilon}\right)\right)\left(\sum_{s = 1}^S \beta_0^s + \rho_0^s\right)\right\rceil} \\
& & + \left\lceil \tfrac{2\delta_0\left(\log\log\left(\tfrac{2+2\delta_0}{\epsilon}\right) + \log\left(\tfrac{2+2\delta_0}{\epsilon}\right) + \log\left(\sum_{s=1}^S \tfrac{\|\mu_0^s - \bar{\mu}^s\|^2}{\rho_0^s} + \tfrac{\|\lambda_0^s - \bar{\lambda}^s\|^2}{\beta_0^s}\right)\right)}{\log(\gamma)}\right. \\ 
& & \left. + \left(\sum_{s=1}^S \tfrac{\|\mu_0^s - \bar{\mu}^s\|^2 + \|\bar{\mu}^s\|^2}{\rho_0^s} + \tfrac{\|\lambda_0^s - \bar{\lambda}^s\|^2 + \|\bar{\lambda}^s\|^2}{\beta_0^s}\right)\right\rceil. 
\end{eqnarray*}
Since $\beta_k^s = \gamma^k \beta_0^s$ and $\rho_k^s = \gamma^k \rho_0^s$ for all $s \in [S]$ and $\delta_k = \tfrac{\delta_0}{\gamma^k}$, Lemma~\ref{Lemma:subroutine-AL} guarantees that the required number of gradient evaluations at the $k^\textnormal{th}$ iteration is bounded by
\begin{equation*}
N_k = \left\{
\begin{array}{ll}
O\left(\gamma^k\left(\sqrt{\tfrac{\sum_{s=1}^S (\beta_0^s\|A_s\|^2 + \rho_0^s\|E_s\|^2) D^2}{\delta_0}} + \tfrac{\sqrt{N}\ell_\theta D^2}{\delta_0}\right)\right), & \textnormal{if } \alpha = 0, \\
O\left(\left(\gamma^{\tfrac{k}{2}}\sqrt{\tfrac{\sum_{s=1}^S (\beta_0^s\|A_s\|^2 + \rho_0^s\|E_s\|^2)}{\alpha}} + \tfrac{\sqrt{N}\ell_\theta}{\alpha}\right)\log\left(\tfrac{\gamma^k(\sqrt{N}\ell_\theta+\ell_G)D^2}{\delta_0}\right)\right), & \textnormal{if } \alpha > 0. \\
\end{array}
\right. 
\end{equation*}
Therefore, we conclude that the total number of gradient evaluations required to return an $\epsilon$-solution of an NGNEP is
\begin{equation*}
N_{\textnormal{grad}} = \sum_{k=1}^{T-1} N_k = \left\{
\begin{array}{ll}
O\left(\epsilon^{-1}\log(1/\epsilon)\right), & \textnormal{if } \mu=0, \\
O\left(\epsilon^{-1/2}\log(1/\epsilon)\right), & \textnormal{if } \mu>0. \\
\end{array}
\right. 
\end{equation*}
This completes the proof.

\end{document}